\documentclass[a4paper,reqno]{amsart}
\usepackage{amsfonts}

\usepackage{pdfsync}
\usepackage{ifthen}
\usepackage{amsmath,paralist}
\usepackage{amsthm}
\usepackage{graphicx,eucal}
\usepackage{graphpap,latexsym,epsf,bm}
\usepackage{color}
\usepackage{hyperref}
\usepackage{amssymb,mathrsfs,enumerate}
\usepackage[normalem]{ulem}
\usepackage{ifthen}
\usepackage{hyperref}
\usepackage{amsmath,amsthm}
\usepackage{amsfonts}
\usepackage[american]{babel}
 \usepackage[T1]{fontenc}
 \usepackage[utf8]{inputenc}
\usepackage{mathrsfs}
\usepackage{psfrag}
\usepackage{epsfig,inputenc}
\usepackage{graphpap,latexsym,epsf}
\usepackage{color}
\usepackage{amssymb,eucal,paralist,color,enumerate}
\usepackage{graphicx}

\usepackage[eulergreek]{sansmath}

\title[Visco-energetic solutions for brittle fracture]
{Visco-energetic solutions for a model of crack growth \\ in brittle  materials}
\author[Gianni Dal Maso]{Gianni Dal Maso}
\address[Gianni Dal Maso]{SISSA, Via  Bonomea 265, 34136  Trieste,
Italy}
\email[Gianni Dal Maso]{dalmaso@sissa.it}
\author[Riccarda Rossi]{Riccarda Rossi}
\address[Riccarda Rossi]{DIMI, Universit\`a degli studi di  Brescia, Via Branze 38, 25133, Brescia, Italy}
\email[Riccarda Rossi]{riccarda.rossi@unibs.it}
\author[Giuseppe Savar\'e]{Giuseppe Savar\'e}
\address[Giuseppe Savar\'e]{Department of Decision Sciences, Universit\`a Bocconi, Via Roentgen 1, 20136, Milano, Italy}
\email[Giuseppe Savar\'e]{giuseppe.savare@unibocconi.it}

\author[Rodica Toader]{Rodica Toader}
\address[Rodica Toader]{DMIF, Universit\`a degli studi di  Udine, Via delle Scienze 206, 33100, Udine, Italy}
\email[Rodica Toader]{rodica.toader@uniud.it}

\newcommand{\R}{\mathbb{R}}
\newcommand{\N}{\mathbb{N}}

\newcommand{\Om}{\Omega}

\mathchardef\emptyset="001F

\newtheorem{theorem}{Theorem}[section]

\newtheorem{lemma}[theorem]{Lemma}
\newtheorem{remark}[theorem]{Remark}

\newtheorem{definition}[theorem]{Definition}
\newtheorem{proposition}[theorem]{Proposition}

\newtheorem{corollary}[theorem]{Corollary}
\newtheorem{hypothesis}[theorem]{Hypothesis}

\numberwithin{equation}{section}

\theoremstyle{plain}

\theoremstyle{definition}
\theoremstyle{remark}

\mathchardef\emptyset="001F

\numberwithin{equation}{section}

\newcommand{\down}{\downarrow}

 \def\calE{{\mathcal E}} \def\calF{{\mathcal F}}
 \def\calH{{\mathcal H}} 
  \def\calL{{\mathcal L}}

\def\calV{{\mathcal V}}  
 
  \def\rmC{{\mathrm C}}

\def\dd{\;\!\mathrm{d}} 

\newcommand{\eps}{\varepsilon}
\newcommand{\teta}{\vartheta}
\newcommand{\foraa}{\text{for a.a.\ }}

\newcommand{\Xs}{\Xamb}

\newcommand{\mdn}{\mathsf{d}}

\newcommand{\md}[2]{\mathsf{d}(#1,#2)}
\newcommand{\corr}[2]{\delta(#1,#2)}
\newcommand{\cmd}[2]{\mathsf{D}(#1,#2)}
\newcommand{\cmdn}{\mathsf{D}}

\newcommand{\ene}[2]{\mathcal{E}(#1,#2)}

\newcommand{\corrn}{\delta}

\newcommand{\Vars}[3]{\mathrm{Var}_{#1}(#2,#3)}
\newcommand{\Vari}[4]{\mathrm{Var}_{#1}(#2,[#3,#4])}

\newcommand{\Varname}[1]{\mathrm{Var}_{#1}}

\newcommand{\lli}[2]{{#1}({#2}{-})}

\newcommand{\rli}[2]{{#1}({#2}{+})}
\newcommand{\jump}[1]{\mathrm{J}_{#1}}

\newcommand{\Jvar}[4]{\mathrm{Jmp}_{#1}(#2;[#3,#4])}
\newcommand{\Jvarname}[1]{\mathrm{Jmp}_{#1}}

\newcommand{\vecostname}{\mathsf{c}}
\newcommand{\vecost}[3]{\mathsf{c}(#1,#2,#3)}

\newcommand{\oldvecost}[3]{\widehat{\mathsf{c}}(#1,#2,#3)}

\newcommand{\VE}{\mathrm{VE}}

\newcommand{\stab}[1]{\mathscr{S}_{#1}}

\newcommand{\rstab}[2]{\mathscr{R}(#1,#2)}
\newcommand{\rstabname}{\mathscr{R}}

\newcommand{\rstabt}[2]{\mathscr{R}(#1,#2)}

\newcommand{\hole}[1]{\mathscr{H}(#1)}
\newcommand{\tcost}[4]{\mathrm{Trc}_{#1}(#2,#3,#4)}
\newcommand{\tcostname}[1]{\mathrm{Trc}_{#1}}
\newcommand{\oldtcost}[4]{\widehat{\mathrm{Trc}}_{#1}(#2,#3,#4)}
\newcommand{\oldtcostname}[1]{\widehat{\mathrm{Trc}}_{#1}}
\newcommand{\Gap}[3]{\mathrm{GapVar}_{#1}(#2,#3)}

\newcommand{\piecewiseConstant}[2]{{#1}_{\kern-1pt#2}}
\newcommand{\pwc}{\piecewiseConstant}
\newcommand{\upiecewiseConstant}[2]{\underline{#1}_{\kern-1pt#2}}

\newcommand{\Argmin}{\mathrm{Argmin}}

\definecolor{dmagenta}{rgb}{0.8,0,0.8}

\newcommand{\BLUE}{\color{black}}
\newcommand{\EEE}{\color{black}}

\definecolor{ddcyan}{rgb}{0,0.6,0.9}
\definecolor{dcyan}{rgb}{0,0.4,0.9}
\definecolor{ddmagenta}{rgb}{0.8,0,0.8}
\definecolor{dred}{rgb}{.8,0,0}

\definecolor{vgreen}{rgb}{0.1,0.5,0.2}

\newcommand{\wsigma}{\hacca}

\newcommand{\hacca}{\stackrel{\mathsf{h}}{\to}}
\newcommand{\htop}{\mathsf{h}}

\newcommand{\argmin}{\mathrm{Argmin}}

\newcommand{\huno}{{\mathcal H}^{1}}

\newcommand{\K}{{\mathcal K}(\overline\Om)}
\newcommand{\Kf}{\GGG {\mathcal K}_{\rm fin}(\overline\Om)\nc}
\newcommand{\hd}{\mathsf{h}}
\newcommand{\Kmf}{{\mathcal K}_m(\overline\Om)}
\newcommand{\Kuno}{{\mathcal K}_1(\overline\Om)}

\newcommand{\Khf}{{\mathcal K}_h(\overline\Om)}

\newcommand{\tKmf}{\Kmf} 
\newcommand{\Kpf}{{\mathcal K}_p(\overline\Om)}
\newcommand{\tKpf}{\Kpf}
\newcommand{\tKf}{\mathcal{K}_1(\overline\Om)}
\newcommand{\trial}{J}

\newcommand{\Adm}[2]{\calV(#1,#2)}
\newcommand{\wAdm}[3]{\widetilde{\calV}(#3; #1,#2)}

\newcommand{\Admle}[2]{\calV_{\mathrm{LE}}(#1,#2)}

\newcommand{\dist}[2]{\mathrm{dist}(#1,#2)}

\newcommand{\Xamb}{\Kf}

\newcommand{\RIC}{\color{black}}

\newcommand{\CRIS}{(\Xamb,\calE,\mathsf{h},\mathsf{d},\delta)}
\newcommand{\CRISEL}{(\Xamb,\calE_{\mathrm{LE}},\mathsf{h},\mathsf{d},\delta)}
\newcommand{\Ctran}{\mathrm{C}_{\htop,\mdn}}

\newcommand{\icostname}{\alpha}
\newcommand{\icost}[2]{\icostname(#1,#2)}
\newcommand{\incostname}{\mathsf{e}}
\newcommand{\incost}[3]{\incostname(#1,#2, #3)}

\newcommand{\ATW}[2]{\Delta(#1,#2)}
\newcommand{\ATWname}{\Delta}

\newcommand{\Vfz}[1]{\mathcal{V}_{#1}}

\newcommand{\llim}[2]{{#1}(#2-)}
\newcommand{\rlim}[2]{{#1}(#2+)}

\newcommand{\serifsigma}{{\sigma}}

\newcommand{\sfK}{K}
\newcommand{\serifgamma}{{\gamma}}
\newcommand{\serifGamma}{{\Gamma}}
\newcommand{\serifu}{u}
\newcommand{\sfu}{u}
\newcommand{\sfg}{g}

\newcommand{\uopen}{V}
\newcommand{\uopenp}{V'}

\newcommand{\REVIS}{\color{black}}
\newcommand{\REVISAGAIN}{\color{black}}
\newcommand{\REVISDOUBT}{\color{black}}
\newcommand{\GGG}{\color{black}}
\newcommand{\nc}{\normalcolor}

\begin{document}

\thanks{May 16, 2022}

\begin{abstract}
Visco-energetic solutions have been recently advanced as a new solution concept for rate-inde\-pen\-dent systems, 
alternative to energetic solutions/quasistatic evolutions and balanced viscosity solutions. In the spirit of this novel concept, we revisit the analysis of 
the variational model 
proposed by Francfort and Marigo  
for the quasi-static 
crack growth in brittle materials, 
in the case of antiplane shear. In this context, visco-energetic solutions can be constructed by perturbing the time incremental scheme for quasistatic evolutions by means of 
a viscous correction
inspired by the term introduced by Almgren, Taylor, and Wang in the study of mean curvature flows. 
With our main result we prove 
the existence 
of a visco-energetic solution with a given initial crack. 
We also  show that, if the cracks have a finite number of tips evolving smoothly 
on a given time interval, 
visco-energetic solutions comply with Griffith's criterion. 
\end{abstract}

\maketitle

{\small

\bigskip

\bigskip
 }
\bigskip
\bigskip

\section{Introduction}
 The variational approach to  brittle fracture, based on the classical theory by \textsc{Griffith} \cite{Grif20PRFS},   was initiated more than twenty years ago by 
 \textsc{Francfort} and \textsc{Marigo} \cite{Francfort-Marigo98}  (cf.\ also  \cite{BFM-book}).   
 In these models crack growth 
 results from a trade-off between the competing mechanisms of 
 \begin{itemize}
 \item[-]
 energy conservation, with the \emph{driving energy}  
 given by the stored elastic energy; 
\item[-]
energy dissipation, 
which takes into account the \emph{dissipated energy} 
spent to open the crack.
\end{itemize} 
 In the case of \emph{antiplane shear} the reference configuration is 
 represented by a  bounded, 
 connected, 
 Lipschitz domain $\Omega \subset \R^2$, 
 the displacement $u\colon \Omega \to \R$ is scalar 
  and the cracks are 
  represented by 
  compact subsets $K$ of $\overline \Omega$. 
  The evolution is triggered by a prescribed  time dependent boundary condition $u=g(t)$ on a subset
$ \partial_D \Omega$ of $\partial\Omega$. 
  According to   the model by \textsc{Francfort} and \textsc{Marigo}, 
  for a 
  linearly elastic homogeneous isotropic material
  the competing energy terms are 
 \begin{align}
 &\text{\emph{driving energy:}}\quad
 \label{driving-energy-INTRO}
\calE(t,K) :=  
\min \Big\{ \int_{\Omega{\setminus}K} \tfrac12  |\nabla u|^2 \dd x\, : \  u = g(t)  \text{ on } \partial_D \Omega \setminus K\Big\},
\\
&\text{\emph{dissipated energy:}}\quad
  \huno(K(t){\setminus}K(s)),\text{ where }\huno \text{ denotes the }1\text{-dimensional Hausdorff measure.}  \nonumber
\end{align}
For simplicity the elastic constant and the toughness of the material are normalized to $1$. 

If $\K$ denotes 
the collection of all compact subsets of $\overline\Omega$, \REVIS a quasistatic evolution for the brittle fracture model 
 is a function $K: [0,T]  \to   \K$ 
  fulfilling the following conditions: \EEE
\begin{enumerate}
\item[(I)] \emph{irreversibility}: $K(s) \subset K(t)$ for all $0 \leq s \leq t \leq T$; 
\item[(S)] \emph{stability}: at every $t\in [0,T]$ we have 
\begin{equation}
\label{StabEn}
\ene t{K(t)} \leq \ene t{K'} + \huno(K'{\setminus}K(t)) \qquad \text{for all } K' \in  \K  \text{ with } K' \supset K(t),
\end{equation}
namely the release of potential energy when passing from the current state $K(t)$ to any other state  $K'\in \K$  is smaller than the energy dissipated, that is why \eqref{StabEn} can also be understood as a stability condition;
\item[(E)] \emph{energy-dissipation balance}: 
\begin{equation}
\label{EBEn}
\ene t{K(t)} +\huno(K(t){\setminus}K(s)) = \ene s{K(s)} + \int_s^t \partial_t \ene r{K(r)} \dd r \qquad \text{for all } 0 \leq s \leq t \leq T,
\end{equation}
involving  the stored energy at the process times $s$ and $t$, the energy dissipated in the time interval $[s,t]$, and the work of the external forces  represented by the integral term. 
\end{enumerate}
\par
 \REVIS Condition  (E) was introduced in   the realm of crack propagation in  \cite{DMToa02}; \REVISDOUBT its key role was also highlighted in  \cite{MieThe99MMRI, MielkeTheilLevitas02, MieThe04RIHM},  where the closely related concept of energetic solution to a rate-independent system was advanced.  \REVIS
 In  \cite{DMToa02}, 
   the  notion of quasistatic evolution for  cracks was  analyzed in the antiplane case,
  \EEE imposing a bound on the number of connected components of the crack. \REVIS Thus,  
  the state space $\K$ was replaced \EEE by the space $\Kmf$ of all compact subsets of $\overline\Omega$ with at most $m$ connected components and finite $1$-dimensional Hausdorff measure. The existence  of  quasistatic evolutions  satisfying (I), (S), and (E) was proved  
  by constructing discrete-time approximate solutions: given a  partition $0 = t_\tau^0 <t_\tau^1<\ldots<t_{\tau}^{N_\tau} = T$ of $[0,T]$,  the time incremental minimization scheme
\begin{equation}
\label{TIM-QE}
K^i_\tau \in \Argmin \big\{  \calE(t^i_\tau,K)+\huno(K{\setminus}K_\tau^{i-1})\, : \  K \in \Kmf, \ K \supset K_\tau^{i-1} \big\} \qquad \text{for } i=1,\ldots, N_\tau 
\end{equation}
provides an approximate solution which converges to a continuous-time solution as  the time step tends to $0$. Ever since, the analysis of quasistatic evolutions  for crack propagation models has been extended in several directions, cf.,  e.g., \cite{Chambolle03, Francfort-Larsen, DMFT05,DM-Lazzaroni,Friedrich-Solombrino}. 
\REVIS In fact, \EEE  thanks to their flexibility and robustness, the notion of quasistatic evolution and the parallel concept of \emph{energetic solution}  
  have been extensively applied to a broad class of rate-independent systems (cf.\ \cite{MieRouBOOK} for a survey).
\par
Nonetheless, it has been known for some time that  quasistatic evolutions/energetic solutions have a drawback. Namely, when the energy functional driving the system is nonconvex, such evolutions, as functions of time, may have  `too early'  and  `too long'  jumps  between energy wells,
cf.,  e.g., 
\cite[Ex.\ 6.3]{KnMiZa08ILMC},
and the full characterization of energetic solutions to $1$-dimensional rate-independent systems proved in  \cite{RosSav12}. Essentially,
this is 
due to  the \emph{global} character of the stability 
condition (S), 
 which  involves  the overall energy landscape. 
 These considerations have motivated the quest of alternative weak solvability  notions based on  local, rather than global, minimality. 
 \par
 To our knowledge, the first attempt in this direction dates back to \cite{DT02MQGB}, where,  as an  alternative 
 to \eqref{TIM-QE}, the following 
 time incremental minimization scheme was proposed for brittle fracture growth (still in the two-dimensional antiplane case):
 \begin{equation}
\nonumber
(u_\tau^i, K^i_\tau) \in \Argmin \big\{  \calE(t^i_\tau,K)+ \huno(K{\setminus}K_\tau^{i-1}) + \lambda \| u - u_\tau^{i-1}\|_{L^2(\Omega)}^2 \, : \, K \in \Kmf, \, K \supset K_\tau^{i-1}, u \in H^1(\Omega{\setminus}K) \big\}  
\end{equation}
for $ i=1,\ldots, N_\tau$, 
with $\lambda>0$ a \emph{fixed} constant.    The additional $L^2$-contribution penalizes the $L^2$-distance of the updated discrete displacement $u_\tau^i$ from the previous $u_\tau^{i-1}$,  and thus enforces 
\emph{locality} on the time discrete level. We also record the notion of fracture evolution by local minimality   advanced in \cite{Larsen-epsilon}.
\par
A more general approach to a reformulation of rate-independent evolution devoid of unnatural jumps was pioneered in  \cite{EfeMie06RILS}.  It stemmed  from the idea that rate-independent evolution 
originates in the limit of systems governed by  
two  time scales: the  `fast' inner scale of the system and the `slow', but dominant,  time scale of the external forces. In this perspective, 
 viscous dissipation is negligible during a time interval in which the system evolves continuously,  but it is 
expected to enter into the system behavior at jumps. 
Thus, one  selects those solutions to the original rate-independent system that arise as limits 
of solutions to the viscously regularized system. This procedure leads to an alternative solution concept 
featuring a \emph{local}, in place of a global, stability/minimality condition, and an energy-dissipation balance 
that provides a description of the system behavior at  jumps,  with the possible onset of `viscous behavior'. 
On the one hand, the vanishing-viscosity technique  has been formalized in an abstract setting in 
\cite{MRS12, MRS13} (cf.\ also \cite{Negri14}), that   codified  the properties of these  `vanishing-viscosity solutions'  in
the notion of 
\emph{balanced viscosity} solution. On the other  hand, 
it has been developed and refined in various concrete applications, cf., e.g.,
\cite{DalDesSol11, BabFraMor12, KRZ13, Crismale-Lazzaroni}). 
\par 
As far as  brittle fracture models are concerned, however, 
the 
 vanishing-viscosity approach has been carried out  
 either assuming that the crack path is a priori known, or 
 in specific geometric settings, cf., e.g., \cite{ToaZan06?AVAQ, Cagnetti08, KnMiZa08ILMC, KMZ10-poly,  Lazzaroni-Toader, Lazzaroni-Toader2,  Almi17, CL17fra,  ALL19}). These restrictions are related to the fact that the construction of  balanced viscosity solutions  ultimately relies on the validity of a suitable chain rule for the energy functional driving the system, which seems to be hard  to obtain for  more general fracture models. 
 \par
 That is why, finding an appropriate mathematical formulation for the evolution of brittle fracture as an alternative to the notion of quasistatic evolution, without specific assumptions on the cracks, is still an up-to-date and challenging issue. In this paper we aim to contribute to it by showing how the concept of \emph{visco-energetic} solution to a rate-independent system, recently introduced  in  \cite{SavMin16},  can be successfully applied to the  two-dimensional antiplane model first addressed in \cite{DMToa02}. 
 \par
 As we will see, visco-energetic solutions have a structure in between that of energetic and balanced viscosity solutions. This intermediate character is also apparent in their characterization, obtained for one-dimensional systems in \cite{Minotti17}, in the results of 
 \cite{RS17}, and in their applicability to rate-independent systems in damage, plasticity, and delamination, cf.\ \cite{Rossi2019}. Here we are going to demonstrate that the model by \textsc{Francfort} and \textsc{Marigo}, at least in the versions considered in \cite{DMToa02} and \cite{Chambolle03}, provides yet another example of rate-independent process for which
 visco-energetic solutions are an adequate tool, while the
  balanced viscosity concept fails to apply.
\subsection*{Visco-energetic evolution of brittle fracture}
Visco-energetic  (hereafter often abbreviated as $\VE$) solutions were introduced in  \cite{SavMin16} in the context of an abstract rate-independent system whose state space  is a Hausdorff topological space $(X, \sigma)$, 
endowed with 
\begin{enumerate}
\item  a driving energy functional $\calE\colon  [0,T]\times X \to (-\infty+\infty]$; 
\item  a (possibly asymmetric, quasi-)distance $\mdn\colon  X \times X \to [0,+\infty] $ that encodes the energy dissipation of the system.
\end{enumerate}

In the spirit of \cite{DT02MQGB},  the key idea at the core of $\VE$ concept is to enforce locality by suitably  perturbing the time incremental minimization scheme. In  the general context addressed in  \cite{SavMin16},
this perturbation is obtained by means of 
\begin{enumerate}
 \setcounter{enumi}{2}
 \item
 a \emph{viscous correction}, namely a  lower semicontinuous functional $\delta \colon  X \times X \to [0,+\infty]$, compatible with $\mdn$ in a suitable sense.
 \end{enumerate}
The above elements constitute a  \emph{viscously corrected} rate-independent system 
$(X, \calE, \sigma, \mdn, \corrn)$. 
\RIC We now illustrate how the brittle fracture model  analyzed in  \cite{DMToa02} can be revisited  within the approach of  \cite{SavMin16} by a careful choice of the ambient space, the driving energy, the dissipation quasi-distance\BLUE , \RIC and the viscous correction. 
In what follows, we are going to work 
in the space 
\begin{gather}
\label{ambient-space-intro}
  X =
  \Kf:=\Big\{K\in \K: K\text{ has a finite number of connected
    components},\ \huno(K)<\infty\Big\}\\
    \nonumber
  \text{ endowed with the topology of the Hausdorff distance
    $\htop$}.
\end{gather}
The evolution is driven by the energy functional $\calE : [0,T]\times \Kf \to [0,+\infty)$ defined in \eqref{driving-energy-INTRO}. 

Instead of imposing an a priori  bound on the number of connected components of the crack  \RIC as in  \cite{DMToa02},  \EEE we penalize the nucleation of new connected components by means of the quasi-distance 
$\icostname(K,K')$ defined as the number of connected components of $K'$ disjoint from $K$.
Indeed, 
we fix a constant $\lambda>0$ and we 
consider  the  dissipation distance  $\mdn\colon  \K \times \K \to [0,+\infty]$ defined by
\begin{equation}
\label{diss-dist-intrp}
\md{K}{K'}: = \huno(K'{\setminus}K) +\lambda\,\icost{K}{K'} \qquad 
\end{equation}
if $K\subset K'$, and set equal to $+\infty$ otherwise. The additional term $\lambda\,\icostname(K,K')$  controls the number of connected components: \BLUE indeed, \RIC  the constant $\lambda$ accounts for the \EEE energetic cost of the nucleation of a new connected component of the crack. \RIC This novel contribution
 plays a crucial role  in the proof of the 
 lower semicontinuity of $\mdn$ with respect to the Hausdorff distance,  as stated in  Proposition \ref{lsc-of-d} ahead. \BLUE This \RIC result is
 a generalization of the classical Go\l \c ab Theorem: it provides a lower semicontinuity estimate for  the Hausdorff measure $\huno(K_n{\setminus}H_n)$, for two sequences $(H_n)_n$ and $(K_n)_n$   of compact subsets  of $\overline\Omega$
that may \BLUE possibly \RIC have infinitely many connected components, provided that 
  a bound on the number of connected components of the sets  $(K_n{\setminus}H_n)_n$ is imposed. 
  Another key structural condition  of $\mdn$ is the triangle inequality, which follows from the (non-trivial)  triangle inequality for $\icostname$ proved in Lemma \ref{triangle-ineq}.
  \EEE

 We choose as viscous correction the functional $\corrn\colon  \K\times \K \to [0,+\infty]$ given by
\begin{equation}
\label{delta-intrp}
\REVISDOUBT \delta({K},{K'}): = \ATW{K}{K'}+ \mu\,\icost{K}{K'}, \qquad  \text{with } \ATW{K}{K'}: =  \int_{K'{\setminus}K} \dist x {K} \dd \huno(x),    \qquad \text{if } K\subset K'\,, \EEE
\end{equation}
for some $\mu>0$. \REVIS Unlike $\mathsf{d}$, the functional $\delta$ does not satisfy the triangle inequality.  Indeed, \EEE in  the regular case considered in Section \ref{s:6}, when $K'$ is close to $K$ the integral contribution
\REVISDOUBT $\ATWname$ \EEE
 to $\delta$ is approximately the sum of the squares of the length increments
of the branches of the crack. 
  \REVISDOUBT The functional 
 $\ATWname$ \EEE
 is inspired by the one  \EEE
 introduced by \textsc{Almgren}, \textsc{Taylor} and \textsc{Wang} in \cite{AlmgrenTaylorWang93}, cf.\ also \cite{DG93, LuckhausSturzenhecker95},
where it plays the role of a sort of  \REVIS a squared \EEE  $L^2$-distance between $K$ and $K'$ in the \emph{Minimizing Movement} scheme for the mean curvature flow. 
A similar term has already been used in \cite{Lazzaroni-Toader, Lazzaroni-Toader2} to study a viscosity-driven model of crack growth.
\REVIS The higher order nature of  $\delta$ 
 \REVISDOUBT is in fact revealed 
 by the following inequality, which relates $\ATWname$,  $\huno$, 
and the Hausdorff distance $\mathsf{h}$:
\begin{equation}
\label{higher-order-delta}
 \ATW{K}{K'} \leq   \mathsf{h}(K,K') \, \huno (K'{\setminus}K) \,.
 \end{equation} \EEE
\par
\RIC Let us point out that \REVISDOUBT the term $\ATW{K}{K'}$
  in \eqref{delta-intrp} is well defined for arbitrary compact sets $K$ and $K'$ and,  \RIC unlike in  \cite{Lazzaroni-Toader, Lazzaroni-Toader2}\BLUE, \RIC no structural assumptions are imposed. In particular, no a priori bound \BLUE on \RIC the number of connected components of the cracks \BLUE is \RIC required\BLUE, \RIC thanks to 
 the term $\mu\,\icost{K}{K'}$ (again,   the constant $\mu$ can be interpreted as an additional energetic cost due to the nucleation of a new connected component). \BLUE Like \RIC
for the dissipation distance $\mdn$,  the latter contribution to the viscous correction  indeed
  plays an important part in the proof of 
  the lower semicontinuity properties of $\delta$ (Proposition \ref{Golab3}). \EEE
  \par
   \REVIS
  We mention here that, while we have set all  physical constants equal to $1$ for simplicity, we have preferred to emphasize the dependence of the model on the regularization constants $\lambda$ and $\mu$, which can be 
  chosen arbitrarily small. \EEE 
  \par
Hereafter,  we shall refer to the quintuple
\[
\CRIS  \quad \text{as a  \emph{viscously corrected} rate-independent system for brittle fracture}.
\]
\par
Along the footsteps of \cite{SavMin16} we 
 construct discrete solutions by solving 
 time incremental minimization scheme 
\begin{equation}
\label{IMT}
K^i_\tau \in \Argmin_{K \in \Kf} \left( \calE(t^i_\tau,K) + \mathsf{d}(K^{i-1}_\tau,K) +   \delta(K^{i-1}_\tau, K)  \right) \qquad \text{for } i=1,\ldots, N_\tau,
\end{equation}
with  $K_\tau^0: = K_0$ the initial crack.
\REVISDOUBT Since
\[
\cmd{K}{K'}: = \mathsf{d}(K,K') + \delta(K,K') = \calH^1(K'{\setminus}K)+
\ATW{K}{K'}  + (\lambda{+}\mu)\alpha(K,K') \qquad \text{if } K \subset K',
\]
the minimum problem \eqref{IMT} rephrases in the following form 
\begin{equation}
\label{Dbig-intro}
K^i_\tau \in \Argmin_{K \in \Kf, \  K \supset K_\tau^{i-1}} \left( \calE(t^i_\tau,K) + \calH^1(K{\setminus}K^{i-1}_\tau)  + 
  \ATWname(K^{i-1}_\tau, K) + (\lambda{+}\mu)\alpha(K_\tau^{i-1},K)  \right) 
\end{equation}
which can be immediately compared with the classical minimization scheme \eqref{TIM-QE} for quasistatic evolutions. \EEE
It turns out that,  for every $i=1,\ldots, N_\tau$ \eqref{IMT}
admits  a solution thanks to the aforementioned lower semicontinuity properties of $\mathsf{d}$ and $\delta$,  and of the energy $\calE$. Our main result, \textbf{Theorem \ref{thm:existVEcrack}} ahead, states that there exists a vanishing sequence $(\tau_j)_j$ of time steps  along  which the 
discrete solutions $(\pwc K{\tau_j})_j$, defined by piecewise constant interpolation of the minimizers 
$(K_\tau^i)_{i=1}^{N_\tau}$, converge to a visco-energetic solution of the viscously corrected system $\CRIS$. The latter is a curve $K\colon  [0,T]\to \Xamb$, with jump set $\mathrm{J}_K$,  complying with the following conditions:
\begin{enumerate}
\item[(I)] \emph{irreversibility}: $K(s) \subset K(t)$ for all $0 \leq s \leq t \leq T$; 
\item[($\mathrm{S}_{\mathrm{VE}}$)]
\REVISDOUBT 
 \emph{$(\mathsf{D})$-stability}:  \EEE
   at every $t\in [0,T] \setminus \mathrm{J}_K$ there holds
\[
\begin{aligned}
\ene t{K(t)}  & \leq \ene t{K'} +  \REVISDOUBT \cmd{K(t)}{K'} 
\\
& 
\REVISDOUBT = \ene t{K'} + \calH^1(K'{\setminus}K(t)) + \ATW{K(t)}{K'} + (\lambda{+}\mu)\alpha(K(t),K')  \EEE
  \text{ for all } K' \in \Xamb \text{ with }K'\supset K(t);
 \end{aligned}
\]
\item[($\mathrm{E}_{\mathrm{VE}}$)] the  \emph{energy-dissipation balance}
\[
\ene t{K(t)}  + 
\REVISDOUBT \calH^1(K(t){\setminus}K(s)) +  \EEE
 \Jvar {\vecostname}{K}{s}{t}  = \ene s{K(s)} + \int_s^t \partial_t \ene r{K(r)} \dd r \qquad \text{for all } 0 \leq s \leq t \leq T\,.
\]
\end{enumerate}
 \REVISDOUBT Condition ($\mathrm{E}_{\mathrm{VE}}$)
 features an additional contribution in comparison with the energy-dissipation balance \eqref{EBEn}. Indeed, 
   the term \EEE $ \Jvarname{\vecostname}$ keeps track of  the energy dissipated at jumps and is   defined in terms of the `visco-energetic'  cost $\vecostname$ introduced in \eqref{ve-cost-intro}
below. 
\par
As we have mentioned before, the structure of this solution concept is in between those of quasistatic evolutions 
(cf.\ \eqref{StabEn} \& \eqref{EBEn}) and of balanced viscosity solutions. On the one hand,  the stability condition ($\mathrm{S}_{\mathrm{VE}}$),  though featuring \REVISDOUBT 
 the additional \EEE viscous correction $\delta$ and holding only outside the jump set $\mathrm{J}_K$,
still retains a \emph{global} character.
  On the other hand,  in the energy  balance ($\mathrm{E}_{\mathrm{VE}}$) the 
 dissipation  of energy is not only recorded by the 
 \REVISDOUBT $\calH^1$-length of the opening of the 
 crack  in the interval $[s,t]$ \EEE
  but, like in the case of balanced viscosity solutions, also by   an additional  term 
 that measures the energy dissipated at the jump points of $K$ in 
 $[s,t]$, i.e.\ $ \Jvar {\vecostname}{K}{s}{t}$.   The jump cost $\Jvarname{\vecostname} $ is, in turn, defined in terms of a functional $\vecostname$ 
 obtained by minimizing a suitable transition cost along {\GGG monotone} curves $\teta$ connecting the two end-points $\lli Kt$
	and $\rli Kt$ of the curve $K$ at $t\in \mathrm{J}_K$,  namely
	\begin{equation}
	\label{ve-cost-intro}
	\begin{aligned}
\vecost {t}{\lli K t }{\rli K t}  := \inf \Big \{     \mathrm{Trc}_{\VE}(t;\teta, E)\, :   
E \Subset \R,\
  \teta\in \mathrm{C}(E;\Xamb),\
      \teta(E^\pm) = K(t\pm) \Big\}\,,
    \end{aligned}
  \end{equation}
where $E^-:=\inf E$ and $E^+:=\sup E$. The transition cost 
    \[
  \mathrm{Trc}_{\VE}(t;\teta, E) :  =   
\REVISDOUBT 
\Gap{\ATWname}\teta E   + 
(\lambda {+} \mu)
\Vars {\icostname}{\teta} E 
  + \sum_{s\in E{\setminus}\{\sup E\}} \rstabt t{\teta(s)}
   \]
consists of
\begin{enumerate}
 \item   a quantity related to the `gaps', or `holes', of the set $E$ (which is just  an  arbitrary compact subset of $\R$ and may have a more complicated structure than an interval);  
    \item 
 \REVISDOUBT    $\Varname\alpha$, i.e.\ the total variation of the curve $\teta$ induced by $\alpha$, 
    modulated by the constant $(\lambda{+}\mu) $ also featuring in the minimization scheme \eqref{Dbig-intro}; \EEE  
    \item
a functional $\mathscr{R} \colon  [0,T]\times \Kf \to [0,\infty) $ 
that keeps track of 
 the violation  of the $\VE$-stability condition along the curve $\teta$, 
as it fulfills
$
\rstabt t{\teta(s)}>0  $ 
if and only if $\teta(s)$ does not comply with  ($\mathrm{S}_{\mathrm{VE}}$) at the  process time $t$. 
\end{enumerate}
It is in terms of 
the cost $\vecostname$ that the $
\VE$ concept offers an alternative  description of the system behavior at jumps, in comparison with quasistatic evolutions. Indeed, $\VE$ solutions 
satisfy the jump conditions 
\[
\label{jump-conditions-intro}
\ene t{\lli Kt} - \ene t{\rli Kt}   =   
\REVISDOUBT \calH^1(\rli Kt{\setminus}\lli Kt)  \EEE
+ \vecost t{\lli Kt}{\rli Kt} \qquad \text{for all } t \in \mathrm{J}_K,
\]
cf.\ Proposition \ref{prop:charact-VE} ahead. 
 Thus, the release of elastic energy at a jump point is not only balanced by the length of the crack opening (like it would be for  quasistatic evolutions), but also by the `visco-energetic' cost between the two end-points $\lli Kt$ and $\rli Kt$. 
 \par
 Nonetheless, if, along the footsteps of \cite{DMToa02}, we assume that, on some interval $(\tau_0,\tau_1) \subset [0,T]$ the $\VE$ solution constructed in Theorem \ref{thm:existVEcrack}  has the additional property that the cracks $K(t)$ have a fixed number of tips, which evolve \REVIS in an  absolutely continuous way \EEE  on the interval $(\tau_0,\tau_1)$ along simple and disjoint paths, then we can prove that  Griffith's criterion for crack growth is satisfied, cf.\  Theorem \ref{th:griffith} ahead. This result is  completely analogous to \cite[Thm.\ 8.4]{DMToa02} for quasistatic evolutions.  It reflects the fact that  $\VE$ solutions  essentially differ from  quasistatic evolutions in the description of the energetic behavior of the system at jumps, cf.\ the characterization provided by Proposition \ref{prop:charact-VE} ahead. 
 \paragraph{\bf Plan of the paper.}
In Section \ref{s:2} we recall some preliminary results on Hausdorff convergence, and on the properties of the elastic energy,  proved in   \cite{DMToa02}. Then, in Section \ref{s:3} we introduce the dissipation distance $\mdn$ and the viscous correction $\delta$, and settle their basic properties. Section \ref{s:4} is devoted to the precise definition of visco-energetic solutions and to the statement of our main existence result, Theorem \ref{thm:existVEcrack}. 
In Section \ref{s:5}, the proof of Theorem \ref{thm:existVEcrack} is carried out by showing that the viscously corrected system for brittle fracture $\CRIS$ satisfies the conditions of the  general existence result
\cite[Thm.\ 3.9]{SavMin16}, which thus applies yielding the existence of  $\VE $ solutions.
The main result of  Section \ref{s:6}, Theorem 
\ref{th:griffith}, provides 
a characterization of the behavior at the crack tips of a  $\VE$ solution $K:[0,T]\to \Kf$
 in an interval $(\tau_0,\tau_1)$ during which $K$ evolves \emph{continuously} as a function of time and the crack set $K(t)$ fulfills suitable geometric conditions. Finally, in Section \ref{s:7} we show how, relying on the results from \cite{Chambolle03}, our existence result for $\VE$ solutions can be extended to the planar case of linearized elasticity.
\section{Notation and preliminaries}
\label{s:2}
Throughout the paper, $\Om$ is a fixed bounded connected open subset
of $\R^2$ with Lipschitz
boundary.
As in \cite{DMToa02}, we shall additionally suppose that the boundary of $\Omega$ decomposes into 
\begin{itemize}
\item[-] a Neumann part $\partial_N \Omega$, which is a (possibly empty) relatively open subset of $\partial\Omega$ with a finite number of connected components;
\item[-] the \REVIS (non-trivial) \EEE  Dirichlet part $\partial_D \Omega: = \partial\Omega \setminus \overline{\partial_N \Omega}$; it turns out that  $\partial_D \Omega$ is also a relatively open subset of $\partial\Omega$ with a finite number of connected components.
\end{itemize}
\par
The one-dimensional Hausdorff measure is denoted by $\huno$. The set of all compact subsets
of $\overline\Om$ is denoted by $\K$, whereas 
\GGG
$\Kf$ (resp.~$\Kmf$)
is  the set of all compact subsets $K$ of $\overline\Om$
with $\huno(K)<+\infty$ and 
a finite number of (resp.~at most $m$) connected components. \EEE
\par
\GGG The spaces $\K,\Kf$
are
\nc endowed with the \emph{Hausdorff distance} $\hd$, defined by 
\begin{equation}
\label{hausd-distance}
\hd(H,K):=
\max\big\{ \sup_{x\in H}{\rm dist}(x,K),
\sup_{y\in K}{\rm dist}(y,H)\big\} \qquad \text{for all } H,\, K \in \K,
\end{equation}
where, as usual, $\mathrm{dist}(x,K)  : = \min_{y\in K} |x{-}y| $, with the convention that 
$\mathrm{dist}(x,\emptyset) = \mathrm{diam}(\Omega)$ and $\sup
\emptyset =0$, so that $\hd(\emptyset, K) = 0$ if $K=\emptyset$ and
$\hd(\emptyset, K) =\mathrm{diam}(\Omega)  $ if $K \neq \emptyset$.
\GGG In particular, $\emptyset$ is an isolated point of $\K$. \nc
 Given $(K_n)_n, \, K \subset \K$, we will often write $K_n \hacca K$ whenever $\mathsf{h}(K_n,K)\to 0$. 
  With a slight abuse of notation, the topology induced by 
 the Hausdorff distance is still denoted by $\hd$, and the corresponding  product topology on $[0,T]\times \K$ by $\htop_{\R}$. 
 The following compactness theorem is 
well known (see, e.g., 
\cite[Blaschke's Selection Theorem]{Rog}).
\begin{theorem}\label{thm:compactness}
The metric space $(\K, \hd)$ is compact.
\end{theorem}
 We will also make use of the following result (cf.\ \cite[Cor.\ 3.3, 3.4]{DMToa02}),  
derived from  the Go\l \c ab theorem (cf., e.g., \cite[Thm.\ 10.19]{Morel-Solimini})
\REVIS and extending the latter, valid in the class $\Kuno$, to the class  $\tKmf$, $m\geq 1$. 
 It 
 shows that  
$\tKmf$ \EEE is closed  w.r.t.\ Hausdorff convergence, and that, with respect to this notion of convergence  
the Hausdorff measure $\huno$ is lower semicontinuous  on $\tKmf$ (while it is not lower semicontinuous on $\K$).
\begin{theorem}\label{Golab2}
Let $m\geq 1$ and  $(K_n)_n \subset \tKmf$. 
\begin{itemize}
\item[(i)] If $\mathsf{h}(K_n, K)\to 0$  as $n \to\infty$ for some $K \in \K$, then $K \in \tKmf$ and 
$$
\huno(K\cap U)\le \liminf_{n\to\infty} \,\huno(K_n\cap U)
$$
for every open set $U\subset \R^2$. 
\item[(ii)] In addition, suppose that $(H_n)_n, H \in \K$, with $\mathsf{h}(H_n, H)\to 0$ as $n\to\infty$. Then,
$$
\huno(K{\setminus} H)\le \liminf_{n\to\infty} \,\huno(K_n{\setminus}H_n)\,.
$$
\end{itemize}
\end{theorem}
\paragraph{\bf Deny-Lions spaces and elastic energy}
\label{ss:D-L}
Along the footsteps of \cite{DMToa02},  we will work with the \emph{Deny-Lions} space \cite{Deny-Lions53} 
\begin{equation}
\label{Deny-Lions}
L^{1,2} (A): = \{ u \in L^2_{\mathrm{loc}}(A)\, : \ \nabla u \in L^2(A;\R^2)\}\,,
\end{equation}
for a given $A\subset \R^2$. If $A$ is bounded with Lipschitz boundary, then $L^{1,2} (A) =H^1(A)$, see  \cite[Prop.\ 2.2]{DMToa02}. 
For a given $g\in H^1(\Omega)$
and a given $K \in \Kf$,
 let us now introduce the space of \emph{admissible displacements}
\begin{equation}
\label{def-admis-u}
 \Adm {g}K: = \{ v \in L^{1,2}(\Omega{\setminus}K)\, :  \ v = g \  \ \text{on }  \partial_D \Omega \setminus K\}\,.
\end{equation}
In \eqref{def-admis-u} the equality $v=g$ is to be interpreted in the sense of traces. Note that the trace of $v$ 
on $\partial_D \Omega \setminus K$ is well defined, since $\partial\Omega$ is Lipschitz (see
e.g.\  \cite[Prop.\ 2.2]{DMToa02}). 
 \par
 As discussed in 
\cite[Sec.\ 4]{DMToa02},  the minimum problem 
\begin{equation}
\label{min-u}
\min_{v \in  \Adm {g}K} \int_{\Omega{\setminus}K} \tfrac12 |\nabla v|^2 \dd x  \qquad \text{ has a solution}. 
\end{equation}
We mention in advance that this minimization problem will be involved in the definition of the energy functional $\calE$ driving our system. It may happen that the minimizer is not unique, 
but, by strict convexity, any  two minimizers have the same gradient on $\Omega{\setminus}K$. The following result, proved in \cite[Thm.\ 5.1]{DMToa02}, shows the continuous dependence of these gradients on the set $K$ and on the boundary datum $g$, 
and will ensure the continuity properties of the energy functional $\calE$. 
\begin{proposition}
\label{prop:continuity-minimizers}
Let $m\geq 1$ and \REVIS let $(K_n)_n,\, K \in \tKmf$ \EEE fulfill $\sup_{n\in \N} \huno(K_n)<+\infty$ and  $\mathsf{h}(K_n,K) \to 0$ as $n\to\infty$. Let $(g_n)_n,\, g  \in H^1(\Omega)$  with $g_n\to g$ strongly in $H^1(\Omega)$. 
Let $(u_n)_n,\, u$ fulfill
\[
u_n \in \Argmin_{v \in  \Adm {g_n}{K_n}} \int_{\Omega{\setminus}K_n} \tfrac12 |\nabla v|^2 \dd x \qquad \text{for all } n \in \N, \qquad  u \in \Argmin_{v \in  \Adm {g}K} \int_{\Omega{\setminus}K} \tfrac12 |\nabla v|^2 \dd x\,.
\]
Then, 
\begin{equation}
\label{strong-converg-grad}
\nabla u_n \to \nabla u \qquad \text{ as $n\to\infty$ in $L^2(\Omega;\R^2)$},
\end{equation}
where $\nabla u_n$ and $\nabla u$ are regarded as functions defined a.e.\ in $\Omega$. 
\end{proposition}

\section{Setup for visco-energetic solutions for brittle fracture}
\label{s:3}
In this section
 we precisely define the 
\begin{enumerate}
\item   driving \emph{energy functional} $\calE$ (cf.\ \eqref{energy}), 
\item  \emph{dissipation quasi-distance} $\mathsf{d}$ (cf.\ \eqref{dissipation-quasidistance}),
\item  \emph{viscous correction} $\delta$ (cf.\ \eqref{delta}). 
\end{enumerate}
intervening in our notion of visco-energetic evolution of brittle fracture.
 Upon introducing $\calE$, $\mathsf{d}$, and $\delta$, 
 we will also 
settle some of their basic properties;  in particular, those 
underlying the definition of $\VE$ solution.
Further properties will be investigated in Section \ref{s:5} ahead, when carrying out the proof of our existence result Theorem \ref{thm:existVEcrack}.

\subsection*{The energy functional}
Throughout the paper $g \in \mathrm{C}^1([0,T];H^1(\Omega)) $ is a fixed function, whose trace on $\partial_D\Omega$ plays the role of a  time-dependent Dirichlet loading acting on $\partial_D\Omega$. 
The  energy functional $\calE\colon [0,T]\times \Xamb \to [0,+\infty)$ is defined by 
\begin{equation}
\label{energy}
\calE(t,K) :=  
\min_{u\in \Adm {g(t)}K}  \int_{\Omega{\setminus}K} \tfrac12  |\nabla u|^2 \dd x,
\end{equation}
where the space of admissible displacements  is given by \eqref{def-admis-u}. 
As we will see in Proposition \ref{l:ad-A} ahead, $\calE$ is  lower semicontinuous on $[0,T]\times \Xamb$, w.r.t.\ to the product topology
$\htop_\R$ on $[0,T]\times\Xamb$, along sequences with bounded $\mathsf{d}$-distance  from some reference set $K_o\in\Kmf$. A straightforward calculation
shows that  the power functional $\partial_t \calE(t,K)$ exists for every $t \in (0,T)$ and  all $K \in \Xamb$ and that 
\begin{equation}
\label{power}
\partial_t \calE(t,K): = \int_{\Omega{\setminus}K} \nabla \dot{g}(t) \cdot  \nabla u(t) \dd x,
\end{equation} 
 where $\dot{g}(t) \in H^1(\Omega)$ is the 
time derivative of the function $g$ 
and $u(t)\in \Adm {g(t)}K$ is a solution of the minimum problem \eqref{energy};  
the formula for $ \partial_t \calE(t,K)$ is well given since 
$\nabla u(t)$ does not depend on the choice of the minimizer, 
  cf.\ Section \ref{s:2}. 
The upcoming  Proposition \ref{l:ad-A}  will collect all  properties of $\calE$ and $\partial_t \calE$ that are  relevant for our analysis. \EEE
\subsection*{The dissipation quasi-distance}
 Preliminarily, let us introduce a quasi-distance between two sets $H$ and $K$ that keeps track of the (number of) connected components of $K$ disjoint from $H$. 
 More precisely,  $\icostname \colon  \K \times \K \to [0,+\infty] $ is defined in this way:  \REVIS
\begin{equation}
\label{icost}
\begin{aligned}
&
\text{if } H \subset K  &&  \icost{H}{K}: = &&
\hspace{-0.3cm}
\GGG \text{number of the connected components of $K$ that are disjoint from
  $H$,}
\\
& \text{if } H \not\subset K && \icost{H}K: = &&  \hspace{-0.3cm}   +\infty\,.
\end{aligned}
\end{equation} \EEE
\GGG Notice that if $H=\emptyset$ then
$\icost{H}{K}$ is simply the number of the connected components of $K$.
\par
\GGG We say that a function $\beta:\K\times \K\to[0,+\infty]$ fulfills
the triangle inequality if
\begin{equation}
  \label{triangle-ineq}
  \beta( H,L) \leq \beta(H,K)+\beta(K,L)
  \qquad \text{for all } H, \, K,\, L \in \K.
\end{equation}
\nc
Our first result shows that $\icostname$ 
satisfies the triangle inequality.
\begin{lemma}
\label{l:alpha-triangle}
The function $\icostname\colon  \K\times \K \to
[0,+\infty]$
fulfills the triangle inequality \eqref{triangle-ineq}. 
\end{lemma}
\begin{proof}
  It is enough to show \eqref{triangle-ineq}
  \GGG for $\alpha$  \nc
  in the case in which $\icost HK<+\infty$ and $\icost KL<+\infty$
 so that, in particular, $H\subset K\subset L$. 
 Hence, a subfamily of the connected components of $L$ which are disjoint from $H$ consists of connected components of $L$ which are disjoint from $K$. Let $n=  \icost KL$ and suppose that $L$ has at least $j$ connected components,  $L_1, \ldots, L_j$, disjoint from $H$ and that the connected components of $L$ disjoint from
$K$ coincide with  the sets $ L_{j-n+1}, \ldots, L_j$. 
We now have to prove that $\icost HK \geq  j-n$. 
For this, it suffices to consider the connected components
$\{ L_1, \ldots, L_{j-n}\}$ intersecting $K$. For every $\ell \in \{1, \ldots, j-n\}$ we have that $L_\ell $ intersects at least a connected component $K_\ell $ of $K$; since $K\subset L$,  we ultimately have that
$K_\ell \subset L_\ell$. Since $L_\ell \cap H = \emptyset$, also $K_\ell \cap H = \emptyset$. Therefore, each connected component $K_\ell$, $\ell \in\{1, \ldots, j-n\} $, contributes to the number of connected components of $K$
disjoint from $H$, which yields that $\icost HK \geq j -n$. 
Since this holds for every $j\leq \icost HL$, we obtain \eqref{triangle-ineq}.
 \end{proof}
\noindent
Secondly, we prove that 
$\icostname$ is lower semicontinuous w.r.t.\ Hausdorff convergence.
\begin{lemma}
\label{l:icost-1}
For all sequences $(K_n)_n, \, (H_n)_n \subset \K$ we have that
\begin{equation}
\label{lsc-Hausd}
\left( K_n\wsigma K, \ H_n \wsigma H \right) \ \Rightarrow \ \icost{H}{K} \leq \liminf_{n\to\infty} \icost{H_n}{K_n}\,.
\end{equation}
\end{lemma}
\begin{proof}
Preliminarily, we prove the following 
\\
{\em \textbf{Claim}: for every connected component $K^\ell $ of $K$ and every $x \in K^\ell $ there exists a sequence $(K_n^\ell)_n$ such that $K_n^\ell $ is a connected component of $K_n$ for every $n\in \N$ and $K_n^\ell \wsigma \widehat{K}^\ell$ as $n\to \infty$ for some connected  set  $\widehat{K}^\ell \in \K$ such that 
$x \in \widehat{K}^\ell \subset K^\ell$.}
\\
Indeed, since $K_n\wsigma K$, there exists a sequence $(x_n)_n$ such that $x_n\to x$ and $x_n \in K_n$ for every $n\in \N$. Let $\widehat{K}_n^\ell$ be the connected component of $K_n$ containing $x_n$. By the Blaschke Selection and the Go{\l}\c{a}b Theorems,   up to a (not relabeled) subsequence, the sets $(\widehat{K}_n^\ell)_n$ converge to a connected set $\widehat{K}^\ell$, which clearly contains $x$. Thus, 
$\widehat{K}^\ell \subset K^\ell$. 
\par
We are now in a position to prove \eqref{lsc-Hausd}. Indeed, suppose that there are $h$ connected components $K^1, \ldots, K^h$ of $K$ disjoint from $H$. For each $\ell \in \{1,\ldots, h\}$, select a point $x_\ell \in K^\ell$
and consider the  connected sets $(K_n^\ell)_n$ and $\widehat{K}^\ell$ whose existence is ensured by the previously proved claim. Then, 
\[
\forall\, \ell \in \{1, \ldots, h \} \quad \exists\, \bar{n}_\ell \in \N \, \ \forall\, n \geq  \bar{n}_\ell \, : \quad K_n^\ell \cap H_n = \emptyset
\]
(otherwise, we would have $\widehat{K}^\ell \cap H \neq \emptyset$, hence $K^\ell \cap H \neq \emptyset$).  
Thus, setting  $\bar{n} := \max_{\ell \in \{ 1,\ldots, h\}}   \bar{n}_\ell $,
 we have 
 \[
 \icost{H_n}{K_n} \geq h \qquad \text{for all } n \geq \bar{n},
 \]
and \eqref{lsc-Hausd} follows. 
\end{proof}
\noindent
\par
For a given $\lambda>0$, we are now in a position to  define the 
 (asymmetric) \emph{dissipation quasi-distance} $\mathsf{d}$ 
  by 
\begin{equation}
\label{dissipation-quasidistance}
 \mathsf{d}: 
\K \times \K 
\to [0,+\infty], \qquad \mathsf{d}(H,K): = 
 \huno(K{\setminus}H)   + \lambda\,\icost{H}{K}  \,.
\end{equation} 
\begin{remark}
\label{rmk:alpha-contribution}\
\upshape
The contribution $\lambda\,\icostname$ to $\mathsf{d}$ will have the role of controlling 
 the  growth  of  the number of connected components  of the   visco-energetic   fracture evolution $[0,T]\ni t\mapsto K(t) $.
It is  exploiting this term  that we may prove the lower semicontinuity of $\mathsf{d}$ w.r.t.\ to  Hausdorff convergence, in fact extending the 
Go\l \c ab Theorem, cf.\ Proposition \ref{lsc-of-d} ahead. The constant $\lambda$ can be interpreted as the nucleation cost of each new connected component of the crack set. 
\end{remark}

\noindent 
Obviously, 
\begin{subequations}
\label{separation-properties}
\begin{equation}
\label{obvious-separation-property}
\mathsf{d}(K,K)=0 \qquad \text{for every } K \in  \K. 
\end{equation}
On the other hand, $\mathsf{d}$ separates the points of   $\K$,  namely for every  $H,\, K \in \K$ 
\begin{equation}
\label{true-separation}
 \text{$\mathsf{d}(H,K) =0$ implies $H=K$.}
\end{equation}
\end{subequations}
Indeed,  $\mathsf{d}(H,K) =0$ implies 
\REVISAGAIN that $H\subset K$, $\huno(K{\setminus}H) =0$, and 
that all the connected
components of $K$ have non-empty intersection with $H$.
  \REVIS If there exists $x_0 \in K {\setminus} H$, then there exists a
   ball $B(x_0,\rho)$ disjoint from $H$. 
   Let  $\widehat{K}$  be  the connected component of $K$ containing $x_0$; since $\widehat{K} \cap H \neq \emptyset$, $\widehat{K}$  must also contain a point in $  \partial B(x_0,\rho)$, hence     
  $\huno(K{\setminus}H) \geq \huno(\widehat{K}  {\cap} B(x_0,\rho))
  \geq \rho  >0$, \REVISAGAIN in contradiction with   $\mathsf{d}(H,K) =0$. \EEE
  \par
  As an immediate consequence of Lemma~\ref{l:alpha-triangle}
  \GGG and of
  the definition of $\mathsf d$ \nc
  we have:
  \GGG 
 \begin{proposition}
   \label{prop:properties-of-d}
The function $\mathsf{d}\colon  \mathcal{K}(\overline\Omega)\times  \mathcal{K}(\overline\Omega) \to [0,+\infty]$
satisfies the triangle inequality \eqref{triangle-ineq}
and $\Kf=
\GGG\big\{K\in \K:\mathsf d(\emptyset,K)<\infty\big\}$.
\end{proposition} 
\noindent
It is then simple to prove that if
\GGG
$H\in \Kf$ 
and $K \in \K$ 
fulfills $\mathsf d( H,K)<+\infty$, then $K\in \Kf$ as well.
\nc
\begin{lemma}
\label{l:key}
Let $H \in \Khf$ for some $h\geq 1$, and let $K \in \K$ contain $H$
and fulfill $\icost HK=i <+\infty$
\GGG and $\huno(K\setminus H)<\infty$.
\nc
Then,
\begin{equation}
\label{bound-on-connected-components}
K \in \Kmf \qquad \text{with } m=  h+ i\,.
\end{equation}
\end{lemma}
\begin{proof}
\REVIS The bound \eqref{bound-on-connected-components} on the number of connected components of $K$ follows from  the triangle inequality 
\[
\icost{\emptyset}K \leq \icost{\emptyset}H + \icost HK = h + i\,. \qedhere
\] \EEE
\end{proof}
\nc
\noindent
The lower semicontinuity of $\mathsf{d}$ w.r.t.\ 
Hausdorff convergence will be a consequence of  the following result, which in fact  
\REVIS extends the (generalized) version 
of  the Go\l \c ab Theorem proved in Theorem\ \ref{Golab2}. Let us indeed emphasize that, for the localized inequality 
\eqref{2showG***},  we no longer require an a priori bound on the number of connected components of the sets $(K_n)_n$, but only that $\sup_n \icost{H_n}{K_n}<+\infty$.  \EEE
\begin{proposition}
\label{lsc-of-d}
Let $(H_n, K_n)_n \subset \K \times \K$ be a sequence such that $H_n \wsigma H$ and $ K_n \wsigma K$. Suppose that    
the number of connected components of $K_n$ disjoint from $H_n$ is uniformly bounded w.r.t.\ $n\in \N$. 
Then,
\begin{equation}
\label{2showG***}
\huno((K{\setminus}H)\cap U) \leq \liminf_{n\to\infty} \huno((K_n{\setminus}H_n)\cap U).
\end{equation}
for every open set $U\subset\Omega$.
\end{proposition}
\noindent
\REVIS   
First, we will prove  \eqref{2showG***} for $U=\Omega$, cf.\ \eqref{2showG} below. The key idea will be to distinguish between the connected components of the sets $K_n$ that intersect the sets $H_n$, and those that have an empty intersection with them.  We will then be able to apply the known local version of the  Go\l \c ab Theorem in both cases. Secondly, we 
shall point out how the proof of the claim with $U=\Omega$ \EEE can be adapted to yield the localized inequality \eqref{2showG***}. 
\begin{proof} 
Clearly,  passing to a subsequence it is not restrictive to assume that
 there exists $\tilde{k}\in \N$ such that  
 \begin{equation}
 \label{icost-cost}
\text{$K_n$ has   $\tilde{k}$ connected components disjoint from  $H_n$  for all $n\in \N$.}
 \end{equation}
 \smallskip
 We show
 \noindent
 {\sl \bf Claim $1$:} 
\begin{equation}
\label{2showG}
\huno(K{\setminus}H) \leq \liminf_{n\to\infty} \huno(K_n{\setminus}H_n).
\end{equation} \EEE
Clearly, we may suppose that the right-hand side is finite and, up to a further extraction, that 
\begin{equation}
\label{2G}
\lim_{n\to\infty} \huno(K_n{\setminus}H_n) <+\infty\,.
\end{equation} 
For every $n\in \N$, let $\widetilde{\mathcal{C}}_n$ denote the collection of the connected components of $K_n$ that do not intersect $H_n$.  Due to \eqref{icost-cost}, $\widetilde{\mathcal{C}}_n$ has $\tilde{k}$ elements, denoted as $\widetilde{C}_n^1, \ldots, \widetilde{C}_n^{\tilde k}$. Up to a subsequence we may suppose that
\begin{equation}
\label{3G}
 \widetilde{C}_n^{i} \to \widetilde{C}^i \qquad \text{for } i = 1, \ldots, \tilde k,
\end{equation}
for some $ \widetilde{C}^i \in \K$. 
Let us now fix 
two open sets $\uopen$ and $\uopenp$ such that 
$H\subset \uopenp\Subset \uopen$ and let 
 \[
 \eta: = \inf_{x\in \uopenp} \mathrm{dist}(x,\Omega{\setminus}\uopen)>0.
 \]
Since $H_n\wsigma H$ and $H \subset \uopenp$, for $n$ sufficiently large we have $H_n \subset \uopenp$;  for simplicity and without loss of generality, hereafter we shall suppose that $H_n\subset \uopenp$ for every $n\in\N$.
 \par
 Let us now consider the family  $\widehat{\mathcal{C}}_n$  of the connected components $C$ of $K_n$ such that 
 \begin{equation}
 \label{4G}
 C \setminus \uopen \neq \emptyset, \qquad C \cap H_n \neq \emptyset.
 \end{equation}
 We will now show that 
 \begin{equation}
 \label{5G}
 \huno(C{\setminus}\uopenp) \geq \eta \qquad \text{for all } C \in \widehat{\mathcal{C}}_n\,.
 \end{equation}
Indeed, let us consider the $1$-Lipschitz   function $f\colon \Omega \to [0,+\infty)$ defined by  $f(x): = \dist{x}{\uopenp}$. Since $C$ is a connected set, $f(C)$ is an interval. It follows from the second of \eqref{4G} and the fact that $H_n \subset \uopenp$ that $0 \in f(C)$. Furthermore, by the first of \eqref{4G} and since $ \eta= \dist{\uopenp}{\Omega{\setminus}\uopen}$, we also have that $\eta \in f(C)$, so that $[0,\eta ]\subset f(C)$. In particular, 
for every $t\in (0,\eta]$ there exists $x\in C$ such that $f(x) =  d(x,V') = t$, so that $x\in C \setminus \uopenp$. 
Therefore, $(0,\eta] \subset f(C{\setminus}\uopenp)$. Since $f$ is $1$-Lipschitz, we then have 
$\eta \leq \huno(f(C{\setminus}\uopenp)) \leq  \huno(C{\setminus}\uopenp)$, i.e., \eqref{5G}.
\par
Since $H_n \subset \uopenp$, for every $C \in \widehat{\mathcal C}_n$ there holds
$C\setminus \uopenp \subset K_n\setminus H_n$ and therefore
\begin{equation}
\label{6G}
\sum_{C \in \widehat{\mathcal{C}}_n} \huno(C{\setminus}\uopenp) \leq \huno(K_n{\setminus}H_n) \leq M,
\end{equation}
with $M = \sup_n  \huno(K_n{\setminus}H_n)<+\infty$ by \eqref{2G}.
Combining \eqref{5G} and \eqref{6G} we then infer that $ \widehat{\mathcal C}_n$ has at most $\frac M{\eta}$ elements. We may then suppose, up to a subsequence, that $ \widehat{\mathcal C}_n$  consists of exactly $\hat{k}\in \N$ elements $\widehat{C}_n^1, \ldots, \widehat{C}_n^{\hat k}$ for every $n$. There exist compact and connected subsets 
 $\widehat{C}^j \in \K$, $j=1, \ldots, \hat{k}$, such that
$
 \widehat{C}_n^j \wsigma \widehat{C}^j
 $
 as $n\to\infty$; moreover, it follows from \eqref{4G} that 
 \begin{equation}
 \label{7G}
  \widehat{C}^j
\setminus \uopen \neq \emptyset, \qquad   \widehat{C}^j \cap H \neq \emptyset.
 \end{equation}
 We will now prove that
 \begin{equation}
 \label{8G}
 K \setminus \overline{\uopen} \subset \bigcup_{i=1}^{\tilde k} (\widetilde{C}^i{\setminus}\overline \uopen) \cup 
 \bigcup_{j=1}^{\hat k} (\widehat{C}^j{\setminus}\overline \uopen)\,.
 \end{equation}
 Indeed, for every $x\in K \setminus \overline \uopen$ there exists a sequence $(x_n)_n$ such that $x_n\to x$ as $n\to\infty$  and $x_n \in K_n \setminus \overline \uopen$ for sufficiently big $n$. Let $C_n$ be the connected component of $K_n$ containing $x_n$. There exists $C^* \in \Kuno$ such that, up to a subsequence, $C_n \wsigma C^*$, so that $x \in C^*$. 
Now, for every $n\in \N$ we either have $C_n \cap H_n = \emptyset$ or $C_n \cap H_n \neq \emptyset$. In the former case, $C_n \in \widetilde{\mathcal C}_n$. In the latter case, since $x_n \in C_n \setminus \uopen \neq \emptyset$, we have
$C_n \in \widehat{\mathcal C}_n$. If $C_n \in \widetilde{\mathcal{C}}_n = \{ \widetilde{C}_n^1, \ldots, \widetilde{C}_n^{\tilde k} \}$ for infinitely many indexes $n$, there exists $i_0 \in \{1, \ldots, \tilde k\}$ such that 
$C_n = \widetilde{C}_n^{i_0}$ for infinitely many $n$ so  that   $C^*= \widetilde{C}^{i_0}$ and, ultimately, $x\in  \widetilde{C}^{i_0} \subset \cup_{i=1}^{\tilde k}  \widetilde{C}^{i}$. 
If $C_n \in \widehat{\mathcal{C}}_n = \{ \widehat{C}_n^1, \ldots, \widehat{C}_n^{\hat{k}}\}$ for infinitely many indexes, then there exists $j_0 \in \{1, \ldots, \hat{k}\}$ such that 
$C_n = \widehat{C}_n^{j_0}$ for infinitely many $n$, so that  $C^*= \widehat{C}^{j_0}$ and thus
$x\in  \widehat{C}^{j_0} \subset \cup_{j=1}^{\hat{k}}  \widehat{C}^{j}$. We have thus proved \eqref{8G}. 
\par
By the local version of the Go\l \c ab Theorem (cf.\ Theorem \ref{Golab2}{i}), we have 
\begin{equation}
\label{9G}
\begin{aligned}
& \huno(\widetilde{C}^i{\setminus}\overline{\uopen}) \leq \liminf_{n\to\infty} \huno(\widetilde{C}_n^i{\setminus}\overline{\uopen}) \quad \text{for all } i=1, \ldots, \tilde{k} \quad \text{and} \\
& \huno(\widehat{C}^j{\setminus}\overline{\uopen}) \leq \liminf_{n\to\infty} \huno(\widehat{C}_n^j{\setminus}\overline{\uopen}) \quad \text{for all } j=1, \ldots, \hat{k}.
\end{aligned}
\end{equation}
Hence, from \eqref{8G} and \eqref{9G} we deduce that
\begin{equation}
\label{10G}
\begin{aligned}
\huno( K {\setminus} \overline{\uopen} ) \leq \sum_{i=1}^{\tilde k} \huno(\widetilde{C}^i{\setminus}\overline \uopen)
+  \sum_{j=1}^{\hat{k}} \huno(\widehat{C}^j{\setminus}\overline \uopen)
&
\leq
 \sum_{i=1}^{\tilde k} \liminf_{n\to\infty}  \huno(\widetilde{C}_n^i{\setminus}\overline \uopen)
 + \sum_{j=1}^{\hat{k}} \liminf_{n\to\infty} \huno(\widehat{C}_n^j{\setminus}\overline \uopen)
\\
& 
\leq 
 \liminf_{n\to\infty} \left(  \sum_{i=1}^{\tilde k} \huno(\widetilde{C}_n^i{\setminus}\overline \uopen){+}  
\sum_{j=1}^{\hat{k}} \huno(\widehat{C}_n^j{\setminus}\overline \uopen) \right)\,.
\end{aligned}
\end{equation}
Now, for every $n$ the  connected components  $\widetilde{C}_n^1, \ldots, \widetilde{C}_n^{\tilde k}$   are pairwise disjoint, and so are the sets
$\widehat{C}_n^1, \ldots, \widehat{C}_n^{\hat k}$. Furthermore, by the very definition of  $\widetilde{\mathcal{C}}_n$ and $\widehat{\mathcal{C}}_n$  we also have that 
$\widetilde{C}_n^i \cap \widehat{C}_n^j = \emptyset$ for every $i=1,\ldots, \tilde k$ and $j=1,\ldots, \hat k$. Therefore, \eqref{10G} leads to 
\begin{equation}
\label{11G}
\huno( K {\setminus} \overline{\uopen} ) \leq \liminf_{n\to\infty} \huno (K_n{\setminus}\overline{\uopen}) \leq\liminf_{n\to\infty} \huno(K_n{\setminus}H_n),
\end{equation}
where the latter inequality holds due to the fact that $H_n\subset \uopen$ (at least for sufficiently large $n$). 
\par
Finally, let $(\uopen_m)_m$ be a sequence of open sets containing $H$ such that $H = \cap_{m=1}^\infty \overline{\uopen}_m$. It follows from \eqref{11G} that $\huno( K {\setminus} \overline{\uopen}_m )  \leq\liminf_{n\to\infty} \huno(K_n{\setminus}H_n)$ for every $m\in \N$ so that, taking the limit as $m\to\infty$ we ultimately have  
$
\huno(K{\setminus}H) \leq \liminf_{n\to\infty} \huno(K_n{\setminus}H_n),
$
i.e.,\ \eqref{2showG}. 
 \smallskip
 \par
 \noindent
 {\sl \bf Claim $2$:} \emph{the localized inequality  \eqref{2showG***}   holds.} It is sufficient to repeat the arguments up to \eqref{8G}, which can be localized, yielding for every open set $U\subset \Omega$
 \[
  (K {\setminus} \overline{\uopen}) \cap U  \subset \bigcup_{i=1}^{\tilde k} ((\widetilde{C}^i{\setminus}\overline \uopen) {\cap} U) \cup 
 \bigcup_{j=1}^{\tilde k} ((\widehat{C}^j{\setminus}\overline \uopen){\cap}U)\,.
 \]
 Then, by the Go\l \c ab Theorem \ref{Golab2} the analogues  of \eqref{9G} hold for 
 $\huno((\widetilde{C}^i{\setminus}\overline \uopen) {\cap} U) $ and $\huno((\widehat{C}^j{\setminus}\overline \uopen){\cap}U),$ yielding the corresponding estimate for $\huno ( (K {\setminus} \overline{\uopen}) {\cap} U )$
 (cf.\ \eqref{10G}). Hence, the analogue of \eqref{11G} holds, i.e.\
 \[
 \huno ( (K {\setminus} \overline{\uopen}) {\cap} U )  \leq \liminf_{n\to\infty} \huno ((K_n{\setminus}H_n){\cap}U).
 \]
From the above inequality it is then possible to infer \eqref{2showG***} by the very same arguments as in Claim $1$.
\\
\REVIS This finishes the proof of Proposition \ref{lsc-of-d}. \EEE
\end{proof}
Recalling that the quasi-distance  $\alpha$ is lower semicontinuous w.r.t.\ Hausdorff convergence by Lemma \ref{l:icost-1}, we immediately deduce the following result from Proposition~\ref{lsc-of-d}.
\begin{corollary}
\label{lsc-of-d-cor}  
The function $\mathsf{d}\colon  \K \times \K \to [0,+\infty]$ is lower semicontinuous w.r.t.\ the Hausdorff distance.
\end{corollary}

 It follows from \eqref{separation-properties},  Proposition \ref{prop:properties-of-d}, and   Corollary  \ref{lsc-of-d-cor}, 
 that   the quasi-distance  $\mathsf{d}$ on  $\Xamb\times \Xamb$ satisfies  the basic conditions required in  \cite[Sec.\ 2.1]{SavMin16}.  

\subsection*{Curves with bounded variation.} 
\REVISDOUBT For a given curve $ K\colon  [0,T]\to \K$,
\GGG a subset $E \subset [0,T]$, and a function
$\beta:\K\times\K\to[0,+\infty]$ satisfying the triangle inequality
\eqref{triangle-ineq}, we define 
 \begin{align}
\label{tot-var}
&
\Vars{\beta}KE : = \sup \Big\{ \sum_{j=1}^N \beta({K(t_{j-1})},{K(t_j)}) \, : \ t_0\le t_1\le \ldots\le t_{N}, \ t_j \in E \text{ for }j=0,\dots,N \Big\}\,,
%
\end{align} 
with the convention that
$\Vars{\beta}K\emptyset =0$. \nc
As we    shall see in the next section, $\VE$ solutions to the viscously corrected system for brittle fracture
satisfy
\begin{equation}
\label{BV-d-curves}
 \Vari{\mdn}K 0T<+\infty\,.
\end{equation} 
\EEE
 Let us now gain further insight into the properties of curves satisfying \eqref{BV-d-curves}.
\par
First of all, from \eqref{BV-d-curves} it clearly follows that $\mathsf{d}(K(s),K(t))<+\infty$ for every $0\leq s \leq t \leq T$, so that 
\begin{equation}
\label{irreversible}
K(s) \subset K(t) \qquad \text{for all } 0 \leq s \leq t \leq T,
\end{equation}
i.e.,\ the function $K$ is   \emph{increasing} w.r.t.\ set inclusion. 

\par 
The next result shows that
\GGG for crack evolutions with values in $\Kf$ and
satisfying  the monotonicity condition \eqref{irreversible}, 
the left and the right limits of $K$ w.r.t.\
$\mathsf{h}$ exist and 
the $\alpha$-variation is concentrated in the jump set;
when the $\alpha$-variation is finite, then the curve is
\emph{$(\htop,\mathsf{d})$-regulated} in the sense of 
\cite[Definition 2.3]{SavMin16}.
\nc
\begin{lemma}
\label{l:K-sigma-regulated}
Let  $ K \colon   [0,T] \to \K$   satisfy
\GGG \eqref{irreversible}. \nc
For $t\in [0,T]$, set
\begin{equation}
\label{left-right}
\lli K t: = \mathrm{cl}\left( \cup_{s<t} K(s) \right), \qquad \rli Kt: = \cap_{s>t} K(s),
\end{equation}
with the conventions $\lli K 0: = K(0)$ and $\rli KT: = K(T)$. Then, 
\begin{equation}
\label{left/right-lims}
\begin{aligned}
&
K(s) \wsigma \lli Kt \quad \text{ as } s \to t_-
&& \text{for all } t \in (0,T],
\\
&
 K(s) \wsigma \rli Kt \quad \text{ as } s \to t_+ && \text{for all } t \in [0,T).
 \end{aligned}
\end{equation}
 Furthermore, there holds $\lli Kt \subset K(t) \subset \rli Kt$ for all $t\in [0,T]$. Let
 $\Theta: = \{ t\in (0,T)\, : \  \lli Kt = K(t) = \rli Kt\}$. Then, 
 \begin{equation}
 \label{properties-Theta}
 \begin{gathered}
 \text{the set $ \jump K: = [0,T]\setminus \Theta$ is at most countable, and  
$
K(t_n) \wsigma K(t)
$ for every $t\in \Theta$ and every $t_n  \to t$}.
\end{gathered}
\end{equation}
\GGG If in addition $K$ takes values in $\Kf$ then
\begin{equation}
\label{only-jump}
\Vari{\icostname}K0t= \sum_{s \in \jump K {\cap} [0,t[} \left(
  \icost{\llim Ks}{K(s)}{+} \icost{K(s)}{\rlim K s} \right)
+\icost{\llim Kt}{K(t)}
\quad\text{for all } t \in [0,T]\,.
\end{equation}
If $\Vari{\icostname}K0T<\infty$ then
the set
\begin{equation}
  \label{jump-alpha}
  \jump{K,\alpha}:=\Big\{s\in \jump K: \icost{\llim Ks}{K(s)}{+} \icost{K(s)}{\rlim K
    s}>0\Big\}\text{ is finite }
\end{equation}
and
\begin{equation}
\label{def:BVsigma-d}
\begin{aligned}
&
\lim_{s\to t_-} \md{K(s)}{\lli Kt} =\lim_{s\to t_-}  \calH^1(\lli Kt{\setminus} K(s)) =\lim_{s\to t_-}  \icost{K(s)}{\lli Kt} =0 && \text{for all } t \in (0,T],
\\
& 
\lim_{s\to t_+} \md{K(s)}{\rli Kt} =\lim_{s\to t_+}  \calH^1(K(s){\setminus} \rli Kt) =\lim_{s\to t_+}  \icost{\rli Kt)}{K(s)}=0
 && \text{for all } t \in [0,T).
 \end{aligned}
 \end{equation}
\end{lemma}
\begin{proof}
Properties \eqref{left/right-lims} are an immediate consequence of
definitions \eqref{left-right} and of the \EEE  definition of
Hausdorff distance, while \eqref{properties-Theta} has been proved in
\cite[Prop.\ 6.1]{DMToa02}.
\par
\GGG
In order to prove \eqref{only-jump} and
\eqref{jump-alpha} in the case when $K$ takes values in $\Kf$, we introduce 
the functions
 $\Vfz{\icostname}: [0,T]\to [0,+\infty]$ 
and 
$\Vfz{\icostname,\mathrm{jump}}: [0,T]\to [0,+\infty]$  defined by
\[
\Vfz{\icostname}(t): = \Vari{\icostname}K0t \quad\text{ and } \quad
\Vfz{\icostname,\mathrm{jump}}(t): = \sum_{s\in \jump K {\cap} [0,t[} \left( \icost{\llim K s}{K(s)}{+}
\icost{K(s)}{\rlim Ks}\right)+\icost{\llim Kt}{K(t)}\,,
\]
which clearly satisfy $\Vfz{\icostname}\ge
\Vfz{\icostname,\mathrm{jump}}$.
In order to prove the converse inequality, it is not restrictive to
assume $t=T$ and $\Vfz{\icostname,\mathrm{jump}}<+\infty$. Since 
$\alpha $ takes values in $\N$, \eqref{jump-alpha} is immediate, so
that 
we can write $\jump{K,\alpha}=\{s_1,\cdots,s_J\}$ for some $0\le
s_1<s_2<\cdots s_J\le T$ in $\jump K$.
Our thesis follows if we show that for every
interval $I=(s_{j-1},s_j)$, $j\in \{2,\cdots,J\}$, and every
$t_1,t_2\in I$ with $t_1<t_2$ we have $\alpha(K(t_1),K(t_2))=0$.

We argue by contradiction and we assume that there exist $t_1<t_2$ in
some interval $I=(s_{j-1},s_j)$ such that $\alpha(K(t_1),K(t_2))\ge
1$. This means that $K(t_2)$ contains a connected component $C$ such
that
$C\cap K(t_1)=\emptyset$. $C$ is compact and
$C'=K(t_2)\setminus C$ is compact as well, since $K(t_2)$ has a finite number of
connected components. There exists a $\rho>0$ such that $C'\cap
C^\rho=\emptyset$,
where $C^\rho:=\{x\in \overline\Omega:\dist xC<\rho\}$.

We now consider the monotone family of compact sets $C(t):=K(t)\cap C$
and we set $r:=\inf\{t\in [t_1,t_2]:C(t)\neq\emptyset\}$.
If $r\in (t_1,t_2)$, then  it is easy to check that
$\rlim C{r}=\rlim K{r}\cap C\neq\emptyset$ and
$\llim K{r}\cap C=\emptyset$ since $K(t)\subset C'$ for every
$t<r$. This implies that $\llim K{r}\neq \rlim K{r}$
and
$\icost{\llim K{r}}{\rlim K{r}}\ge 1$, hence $r\in \jump{K,\alpha}$. This is  
a contradiction, since, by construction, $r\in [t_1,t_2]\subset[0,T]\setminus  \jump{K,\alpha}$.
If $r=t_2$ we use the same argument by replacing $\rlim C{t_2}$ with
$C(t_2)$;
similarly, if $r=t_1$ we replace $\llim K{t_1}$ with $K(t_1)$.

\GGG If $\Vari{\icostname}K0T$ is finite then also $\Vari{\mathsf d}K0T$ is
finite and 
we can
\nc
now exploit \eqref{BV-d-curves} in order to check 
\eqref{def:BVsigma-d} for $\lli Kt$ (the proof of the assertion for $\rli Kt$ follows the same lines).
For every $s\in [0,T]$, let $V(s): = \mathrm{Var}_{\mathsf{d}}(K, [0,s])$. Since $V$ is monotone increasing and
\eqref{BV-d-curves}
holds, we have that 
$V(t-):=\lim_{s\to t_-}V(s)<+\infty$. For every $0<s<s_1<t$ we have that 
$
\mathsf{d}(K(s), K(s_1))\leq V(s_1)-V(s)$. Passing to the limit as $s_1 \to t_-$ and using the semicontinuity of $\mathsf{d}$ (cf.\ Corollary \ref{lsc-of-d-cor}), 
we conclude that 
\[
\mathsf{d}(K(s), K(t-))\leq V(t-)-V(s)\,.
\]
Hence, taking the limit as $s\to t_-$ we conclude that 
$\lim_{s\to t-} \mathsf{d}(K(s), K(t-)) =0$.  This concludes the proof. 
\end{proof}
\GGG
We 
 immediately have the following result.
\begin{corollary}
\label{cor:3.10}
\GGG
Let $K: [0,T]\to \Xamb$ fulfill \eqref{irreversible}.
Then for all $t\in [0,T]$
\begin{equation}
\label{Vard-explicit}
\Vari{\mdn}K0t= \huno(K(t){\setminus}K(0)) + \lambda
\bigg(\sum_{s \in
  \jump K {\cap} [0,t[} \left( \icost{\llim Ks}{K(s)}{+}
  \icost{K(s)}{\rlim K s} \right)+\icost{\llim Kt}{K(t)}\bigg)\,.
\end{equation}
\end{corollary}
\EEE

\subsection*{The viscous corrrection} 
We consider the viscous correction
 $\delta\colon 
 \K \times \K
 \to [0,+\infty]$  defined by 
\begin{equation}
\label{delta}
\delta(H,K):   = \ATW{H}{K}  +  \mu\,\icost{H}{K} \quad \REVISDOUBT \text{with }   \ATW{H}{K} : = \begin{cases}
\int_{K\setminus H}\mathrm{dist}(x,H) \dd \huno(x) & \text{if } H \subset K,
\\
+\infty & \text{otherwise}, \EEE
\end{cases}
\end{equation}
where $\mu>0$ is a prescribed constant, which plays the role of a nucleation cost for each new connected component of the crack set. As before, we adopt the convention that $\mathrm{dist}(x,\emptyset) = \mathrm{diam}(\Omega)$. 
\par
The first property  to be satisfied for $\delta$ to be an \emph{admissible} viscous correction is lower semicontinuity w.r.t.\ Hausdorff convergence. 
As we will see in the proof of Proposition \ref{Golab3}, 
   the contribution of the quasi-distance $\icostname$, modulated by  whatever positive  coefficient $\mu$, 
 has again a key role in ensuring lower semicontinuity, 
as it controls the growth of the number of connected components of $K$ disjoint from $H$. 
\begin{proposition}\label{Golab3}
\begin{enumerate}
\item
Let $H\in\K$ be fixed.  Then  for all $(K_n)_n \subset \K$ 
such that the number of connected components of $K_n$  disjoint from $H$ is uniformly bounded w.r.t.\ $n$, 
 we have that 
\begin{equation}
\label{lsc1}
  K_n \wsigma K  \ \Rightarrow \ 
\REVISDOUBT  \ATW{H}{K} \leq  \liminf_{n\to\infty}   \ATW{H}{K_n}\,. \EEE
\end{equation}
Moreover, 
\begin{equation}
\label{lsc-delta}
\delta(H, K) \leq \liminf_{n\to\infty} \delta(H,K_n).
\end{equation}

\item
For all $(H_n)_n, \, H  \in \K$, and $ (K_n)_n, \, K\in \K$ 
we have
\begin{equation}
\label{prel-lsc-statement}
  \left( H_n \wsigma H, \ K_n \wsigma K \right) \ \Rightarrow \ 
   \delta(H,K) \leq  \liminf_{n\to\infty} \delta(H_n, K_n). 
\end{equation}
\end{enumerate}
\end{proposition}
\begin{proof} $\vartriangleright (1)$:
Let us  observe that for every lower semicontinuous nonnegative function $f\colon \Om\to [0,+\infty)$ we have 
$$
 \int_{K{\setminus}{H}} f \dd \huno(x) \le\liminf_{n \to \infty}\int_{K_n\setminus{ H}}f \dd \huno(x). \EEE
$$ 
Indeed, by the lower semicontinuity of $f$  the set 
$
U_t=\{x\in\Om:f(x)>t\}
$
is open and by Proposition \ref{lsc-of-d}, since the number of connected components of $K_n  $  disjoint from $H$ is uniformly bounded,  we have
$
 \huno((K{\setminus}H){\cap} U_t)\le \liminf_{n\to\infty} \,\huno((K_n{\setminus}H){\cap} U_t)\,.
$
Therefore, by the  Fatou Lemma we have
\[
\begin{aligned}
\int_{K{\setminus} H}f(x) \dd \huno(x) & =\int_0^{+\infty}\huno(\{x:f(x)>t\}{\cap} (K{\setminus} H)) \dd t\\
&
\le\liminf_{n \to \infty} \int_0^{+\infty}\huno(\{x:f(x)>t\}{\cap} (K_n{\setminus} H)) \dd t=\liminf_{n \to \infty} \int_{K_n{\setminus} H}f(x)\dd \huno(x)\,.
\end{aligned}
\]
Choosing $f(x) = \mathrm{dist}(x, H) $
 we obtain \eqref{lsc1}.
\par
Clearly, \eqref{lsc-delta} immediately follows: as we may suppose that $ \liminf_{n\to\infty} \delta(H, K_n)<\infty$, up to the extraction of a further subsequence, we have   that $\sup_n \mu\,\icost{H}{K_n} \leq 
\sup_n \delta(H, K_n)<\infty$; then, it suffices to recall that,  by Lemma
 \ref{l:icost-1},  $\liminf_{n\to\infty}\icost{H}{K_n} \geq \icost{H}K$. 
\par\noindent
 $\vartriangleright (2)$:  
  We may of course suppose $\sup_n\delta(H_n,K_n)<\infty$. By Lemma
 \ref{l:icost-1},  
 $\liminf_{n\to\infty} \icost{H_n}{K_n} \geq \icost HK$. 
   In order to show the lower semicontinuity of the first  contribution to $\delta$,  
let us introduce the set $H^\eps = \{ x\in \overline\Omega\, : \ \mathrm{dist}(x,H)\leq \eps\}$ for every fixed $\eps>0$. 
We have that $H_n\subset H^\eps$ for $n$ large enough; in what follows, for simplicity we will suppose that $H_n\subset H^\eps$ for all  $n$. Thus 
$\mathrm{dist}(x,H^\eps) \leq \mathrm{dist}(x,H_n)$ for all $x\in\overline\Omega$. 
Then,
\[
\int_{K_n\setminus H_n}\mathrm{dist}(x, H^\eps) \dd \huno(x) \leq \int_{K_n\setminus H_n}\mathrm{dist}(x, H_n) \dd \huno(x) \,.
\]
Therefore, 
\[
\begin{aligned}
\liminf_{n\to\infty} \int_{K_n\setminus H_n}\mathrm{dist}(x, H_n) \dd \huno(x) 
& \geq 
\liminf_{n\to\infty} \int_{K_n\setminus H_n}\mathrm{dist}(x, H^\eps) \dd \huno(x)
\\
&  \geq \liminf_{n\to\infty} \int_{K_n\setminus H^\eps }\mathrm{dist}(x, H^\eps) \dd \huno(x) 
\\
 &  {\geq} \int_{K\setminus H^\eps}\mathrm{dist}(x, H^\eps) \dd \huno(x) \geq \int_{K\setminus H}\mathrm{dist}(x, H^\eps) \dd \huno(x)\,,
 \end{aligned}
\] 
where for the last-but-one inequality we have applied 
$\eqref{lsc1}$. This is possible since the boundedness of  
$\icost{H_n}{K_n}$
and the inclusions  $H_n \subset H^\eps$ for all $n\in \N$ imply that the  number of connected components of $K_n$ disjoint from $H^\eps$ is uniformly bounded w.r.t.\ $n\in \N$. 
Since $\eps>0$ is arbitrary, we may pass to the limit 
as $\eps \down 0$ via the Fatou Lemma  to obtain \eqref{prel-lsc-statement}. \EEE
\end{proof}
In Section \ref{s:5} we will gain further insight into the properties of $\delta$, cf.\ Proposition \ref{l:ad-B} ahead.
\REVIS Therein, we will use estimate \eqref{higher-order-delta} to show 
 that $\delta$ is of higher order with respect to $\mathsf{d}$ in a precise, technical sense.   \EEE
\section{$\VE$ solutions: definition and main results}
\label{s:4}
In this Section we give the Definition of visco-energetic solution to the system $\CRIS$ for brittle fracture and state our main existence result.

\subsection{Definition of $\VE$ solution}
\label{ss:4.1}
 The definition of the $\VE$ concept (cf.\ Def.\ \ref{def:VEcrack} ahead) hinges on a notion of stability, introduced in Def.\ \ref{def:stability} below, that involves both the dissipation quasi-distance $\mdn$ and its viscous correction $\delta$, and on an energy-dissipation distance featuring a  cost that suitably measures the energy dissipated at jumps.
\paragraph{\bf Stable sets in the visco-energetic sense}
With the viscous correction $\delta$ defined in  \eqref{delta} at hand, 
we introduce the `corrected' dissipation $\cmdn \colon  
\mathcal{K}(\overline\Omega) \times \mathcal{K}(\overline\Omega) \to [0,+\infty]$ 
\[
\cmd{H}{K} : = \md{H}{K} + \corr{H}{K} =
\begin{cases}
 \calH^1(K{\setminus}H) +  
 \ATW{H}{K}
 + (\lambda+\mu) \icost HK   & \text{if } H \subset K,
 \\
+ \infty & \text{otherwise}.
 \end{cases}
\]
We are now in a position to introduce the notion of stability in the visco-energetic sense. 
\begin{definition}
	\label{def:stability}
	Let $Q\geq 0$. We say that $(t,K)\in [0,T]\times \Xamb$ 
	is \emph{$(\cmdn,Q)$-stable} if it satisfies
	\begin{equation}
	\label{Qstable}
	\ene t{K} \leq \ene t{K'} + \cmd{K}{K'}  +Q \qquad \text{for all } K'\in  \Xamb. 
	\end{equation}
	If $Q=0$, we will simply say that $(t,K)$ is $\cmdn$-stable. We denote by $\mathscr{S}_{\cmdn}$ the collection of all the $\cmdn$-stable points, and by
	 $\mathscr{S}_{\cmdn}(t) : = \{ K \in \Xamb\, : \ (t,K) \in \mathscr{S}_{\cmdn}\}$ its section at the  process  time $t\in [0,T]$. Analogously, with the symbols  $\mathscr{S}_{\cmdn}^Q$
	 and  $\mathscr{S}_{\cmdn}^Q(t)$ we will denote    the  $(\cmdn,Q)$-stable sets and their sections.
\end{definition}
\noindent
  We   introduce the \emph{residual stability function}
  $\rstabname\colon  [0,T]\times \Xamb \to [0,+\infty]$  via 
 \begin{equation}
 \label{residual-stability-function}
 \begin{gathered}
 \rstabt tK:  = \sup_{K'\in \Xamb} \left\{ \ene tK {-} \ene t{K'} {-} \cmd{K}{K'} \right\} = \ene tK -\mathscr{Y}(t,K)  \quad \text{with}
\\
 \mathscr{Y}(t,K)  :=  \inf_{K'\in \Xamb} \left( \ene t{K'} {+} \cmd{K}{K'}\right)\,.
 \end{gathered}
\end{equation} 
By the properties of $\calE$ (cf.\ Section \ref{ss:5.1-old}) and the lower semicontinuity of $\mdn$ and $\delta$, 
\begin{equation}
\label{minimal-set}
M(t,K): =  \mathrm{Argmin}_{K'\in \Xamb}\left( \ene t{K'} {+} \cmd{K}{K'}\right) \neq \emptyset.
\end{equation} 
 Observe that $\mathscr{R}$  in fact records the failure of the stability condition at a given point  $(t,K) \in [0,T]\times \Xamb$, since
 \begin{equation}
 \label{propR}
 \begin{aligned}
 \rstabt tK \geq 0   \text{ for all } (t,K) \in [0,T]\times \Xamb, \quad \text{with }
  \rstabt tK=0  \text{ if and only if } (t,K) \in \stab {\cmdn}\,.
 \end{aligned}
 \end{equation}
 Furthermore, $\mathscr{R}$ is lower semicontinuous w.r.t.\ the product topology $\htop_\R$  
 on $[0,T]\times \Xamb$
 if and only if for every $Q\geq0$ the $(\cmdn,Q)$-stable sets
  are $\htop_\R$-closed. 

 \paragraph{\bf The visco-energetic cost $\vecostname$}
It is defined by minimizing 
  a suitable  \emph{transition cost} functional over a class 
 of 
 curves, connecting the left- and right-limits $\lli K t$ and $ \rli K t$ at a jump point $t\in \jump K$. 
 Such curves  
 are in general defined on a compact subset $E\subset \R$ with  a possibly more complicated structure than  that of an interval. 
 They are continuous w.r.t.\  the Hausdorff topology $\htop$, increasing in the sense of \eqref{irreversible}, and 
  satisfying the following  additional continuity condition w.r.t.\ the dissipation distance $\mdn$
  \begin{equation}
 \label{1sided-continuity}
 \forall\, \eps>0 \ \exists\, \eta>0 \, : \ \mdn(\teta(s_0),\teta(s_1)) \leq \eps \quad \text{ for all } s_0,s_1 \in E \text{ with } s_0\leq s_1 \leq s_0+\eta\,.
 \end{equation} 
 Such conditions define the space 
 \begin{equation}
 \label{space-C-sigma-d}
 \begin{aligned}
 \Ctran(E;\Xamb) : = \{ \teta \in \mathrm{C}(E; (\Xamb,\htop))\, : \ & 
 \teta \text{ fulfills \GGG \eqref{irreversible} and
   \nc \eqref{1sided-continuity}} \}\,.
\end{aligned}
 \end{equation}
\REVISDOUBT We are now in a position to introduce the transition cost
$\tcostname{\VE}$
 that will give rise to the visco-energetic cost. We mention in advance that our definition 
 of $\tcostname{\VE}$ differs from  the definition in \cite[Def.\ 3.5]{SavMin16}:  as we will point out later on, 
 we have indeed tailored 
 the structure of  $\tcostname{\VE}$  to the present setup of crack propagation. Nonetheless, the notion of $\VE$ solution arising from our own definition of $\tcostname{\VE}$  ultimately coincides with the notion of evolution proposed in \cite{SavMin16}; we will detail this in Sec.\ \ref{ss:revisit} below. \EEE
\begin{definition}
\label{def-VE-trans-cost}
Let $E$ be a compact subset of $\R$, let  \REVIS $E^- := \min E$,  $ E^+ :=\max E$, \EEE and let
 $\teta \in  \Ctran(E;\Xamb)  $. 
For every $t\in [0,T]$ we define the \emph{transition cost function}
\begin{equation}
\label{ve-tcost}
\REVISDOUBT 
\tcost{\VE}{t}{\teta}{E} : =
\Gap{\ATWname}{\teta} E  +
(\lambda {+}\mu) \Vars {\icostname}{\teta} E \EEE
+ 
 \sum_{s\in E{\setminus}\{E^+\}}
 \rstab t{\teta(s)} \quad \text{with }
\end{equation}
\begin{enumerate}
\item \REVISDOUBT  $\Gap{\ATWname}\teta E
\colon = \sum_{I\in \hole E} \ATW{\teta(I^-)}{\teta(I^+)} $,  \EEE where $I^- := \inf I$,  $I^+:=\sup I$, and $\hole E$ is 
the collection  of the connected components of  $[E^-,E^+]\setminus E$;
\item $ \Vars {\icostname}\teta E $  the  $\icostname$-total variation of the curve $\teta$,  \REVISDOUBT cf.\ 
  \GGG \eqref{tot-var},  \EEE
\item the (possibly infinite) sum
\[
 \sum_{s\in E\setminus \{E^+\}}
  \rstab t{\teta(s)}: =  \begin{cases}
 \sup \{  \sum_{s\in P} \rstab t{\teta(s)} :  \, P \text{ finite, } P\subset E\setminus\{ E^+\}\} 
 &\text{ if }  
 E \neq \emptyset,
 \\
 0 &\text{ otherwise}.
 \end{cases}
\]
\end{enumerate}
\end{definition}
\noindent
\GGG
It is worth noticing that the contribution of $\Vars\icostname\teta E$
to the transition cost $\tcost{\VE}{t}{\teta}{E}$
is concentrated on the gaps of $E$.
 \begin{lemma}
 \label{l:GapVar1}
 Let  $E \subset \R$ be  compact subset and  $\teta \in  \Ctran(E;\Xamb)$. Then,
 \begin{equation}
 \label{total&gap-coincide}
 \Vars{\icostname}\teta E = \Gap{\icostname}\teta E
 =\sum_{I\in \mathscr H(E)}\icost{\teta(I^-)}{\teta(I^+)}.
 \end{equation}
 \end{lemma}
 \begin{proof}
   We consider the extension $\hat \teta:[E^-,E^+]\to \Xamb$ obtained
   by setting $\hat \teta(s):= \teta(I^-)$
   for every $s\in I$, $I\in \mathscr H(E)$.
   It is easy to check that $\hat\teta$ satisfies
   \eqref{irreversible},
   $\Vari{\icostname}{\hat\teta}{E^-}{E^+}=\Vars{\icostname}\teta E$,
   and
   \[\jump{\hat\teta}\subset \{I^+:I\in \mathscr H(E)\}
     \quad\text{with}\quad
     \llim{\hat\teta}{I^+}=\teta(I^-),\quad
     \hat\teta(I^+)=\rlim{\hat\teta}{I^+}=
   \teta(I^+).\]
 Then,   \eqref{total&gap-coincide}  follows by \eqref{only-jump}.
   \end{proof}
   \par
   \nc
We can now introduce the  \emph{visco-energetic jump dissipation cost} $\vecostname\colon  [0,T]\times \Xamb \times \Xamb \to [0,+\infty]$ 
between the two end-points of a jump  of an \emph{increasing} curve $K: [0,T] \to \Xamb $. 
Namely, for all $K_-,\, K_+ \in \Xamb$ 
we set
 \begin{equation}
 \label{vecost}
 \begin{aligned}
  \vecost t{K_-}{K_+}: = \inf\{   \tcost{\VE}{t}{\teta}{E}\, :     &  \ \REVIS E  = \overline{E} \EEE \Subset \R, \ \teta  \in  
   \Ctran(E;\Xamb), \ 
   \ \teta(E^-) =K_-, \  \teta(E^+) =K_+ \}, 
   \end{aligned}
\end{equation}
with the convention $\inf \emptyset = +\infty$;
\GGG notice that $\vecost t{K_-}{K_+}=+\infty$ if $K_-\not\subset
K_+$ and
\begin{equation}
  \label{eq:3}
  \vecost t{K_-}{K_+}\ge (\lambda+\mu)\icost{K_-}{K_+}.
\end{equation}
\nc
Along the footsteps of \cite{SavMin16}, 
we define the 
 \emph{jump variation} functional, defined 
along a 
 curve  $K\colon  [0,T] \to \Xamb$ via 
\begin{equation}
\label{jump-Delta-e}
\begin{aligned}
 \Jvar {\vecostname}{K}{t_0}{t_1} : = 
   \vecost{t_0}{K(t_0)}{\rli K{t_0}}  &  + \sum_{t\in \jump K \cap (t_0,t_1)} \left(
\vecost t{\lli Kt}{K(t)} {+} \vecost t{K(t)}{\rli Kt} \right)
 \\
 & +\vecost{t_1}{ \lli K{t_1}}{ K(t_1)} 
  \quad \text{for all } [t_0,t_1] \subset [0,T].
 \end{aligned}
 \end{equation}

\par
We are now in a position to define the concept of visco-energetic solution of the system $\CRIS$,  featuring the 
$\cmdn$-stability condition \eqref{stab-VE} below, required  outside the   jump set  $\jump K$ of the curve $K$,  and  the 
 energy balance \eqref{enbal-VE}, where the energy dissipated  at jumps is recorded by the jump dissipation cost  introduced in  \eqref{jump-Delta-e}. 
\begin{definition}[Visco-energetic solution]
 \label{def:VEcrack}
A  curve  $ K \colon  [0,T]\to \Xamb$ 
 is   a visco-energetic $(\VE)$ solution of the  system $\CRIS$ 
 for brittle fracture if it satisfies
\begin{itemize}
\item[-] the monotonicity condition \eqref{irreversible}; 
\item[-] the $\cmdn$-stability condition
\begin{equation}
\label{stab-VE}
	\begin{aligned}
\ene t{K(t)} &  \leq \ene t{K'}  +
\cmd{K(t)}{K'} 
\quad \text{for all } K' \in  \Xamb\EEE \text{ and all } t \in [0,T]\setminus \jump K,
\end{aligned}
\end{equation}
\item[-]  the  $(\huno,\vecostname)$-energy-dissipation balance  
\begin{equation}
\label{enbal-VE}
\REVISDOUBT \ene t{K(t)}  + \huno(K(t){\setminus}K(0)) 
+  \Jvar {\vecostname}{K}{0}{t} = \ene 0{K(0)}+\int_0^t \partial_t \ene s{K(s)} \dd s \quad \text{for all } t \in [0,T]\,. \EEE
\end{equation}
\end{itemize}
\end{definition}
\GGG Notice that by \eqref{eq:3}
and \eqref{only-jump} a $\VE$ solution $K$ satisfies $\Vari{\mdn}K0T<+\infty$. \REVISAGAIN Moreover, the regularization parameters $\lambda$ and $\mu$ enter the definition  only via the term $ \Jvarname {\vecostname}$,  which in fact depends on their sum $(\lambda{+}\mu)$, as a consequence of \eqref{ve-tcost} and \eqref{vecost}. 

\REVISDOUBT
\subsection{Comparison with the original notion of $\VE$ solution}
\label{ss:revisit} 
We now aim to relate our notion of $\VE$ solution to the concept introduced in \cite{SavMin16}. 
In fact, in Proposition \ref{prop:they-coincide} below we will prove that Definition \ref{def:VEcrack} is a
 reformulation of  \cite[Def.\ 3.7]{SavMin16}. 
\par
Now, the first, significant difference between the setup of Sec.\ \ref{ss:4.1} and that of  \cite{SavMin16} resides in the definition of the 
$\VE$ cost along a transition curve
$\teta\in  \Ctran(E;\Xamb)$, with $E $ a compact subset of $\R$. 
In the approach of   \cite{SavMin16}, such cost, hereby denoted by 
$\oldtcostname{\VE}$, 
features the $\mdn$-total variation of $\teta$ on $E$
(cf.\ \ref{tot-var}), 
 as well as the contribution of the $\delta$-`gap variation', namely
\begin{equation}
\label{old-cost-SM}
\oldtcost{\VE}t\teta E: = 
 \Vars {\mdn}{\teta} E
 + \Gap{\delta}{\teta} E 
  \EEE
+ 
 \sum_{s\in E{\setminus}\{E^+\}}
 \rstab t{\teta(s)} 
\end{equation}
with 
 $\Gap{\delta}\teta E
\colon = \sum_{I\in \hole E} \delta(\teta(I^-),\teta(I^+))$. 
Accordingly, we denote by $\widehat{\vecostname} $ the jump dissipation cost obtained by minimizing 
$\oldtcostname{\VE}$ over all transition curves, cf.\ \eqref{vecost}. 
\par
Secondly, in the notion of $\VE$ solution introduced in \cite[Def.\ 3.7]{SavMin16} 
the energy-dissipation balance records dissipation  in an (apparently) different way, as it has the structure
\begin{equation}
\label{old-enbal-VE}
 \ene t{K(t)}  + 
 \Vari{\mdn} K0t 
+  \Jvar {\incostname}{K}{0}{t} = \ene 0{K(0)}+\int_0^t \partial_t \ene s{K(s)} \dd s
 \quad \text{for all } t \in [0,T]\,,
\end{equation}
with $\GGG  \Jvarname {\incostname}$
the \emph{incremental} jump variation functional induced by 
the `incremental cost' 
\[
\incost{t}{K_-}{K_+} = \oldvecost t{K_-}{K_+}  - \md{K_-}{K_+},
\]
namely
\begin{equation}
\label{IJUMP}
\begin{aligned}
 \Jvar {\incostname}{K}{t_0}{t_1} : = 
   \incost{t_0}{K(t_0)}{\rli K{t_0}}  &  + \sum_{t\in \jump K \cap (t_0,t_1)} \left(
\incost t{\lli Kt}{K(t)} {+} \incost t{K(t)}{\rli Kt} \right)
 \\
 & +\incost{t_1}{ \lli K{t_1}}{ K(t_1)} 
  \quad \text{for all } [t_0,t_1] \subset [0,T].
 \end{aligned}
 \end{equation}
\par
We will show  that, in the present setup for crack propagation, the energy-dissipation balance
 \eqref{old-enbal-VE}
  in fact coincides with \eqref{enbal-VE}
 in Definition \ref{def:VEcrack}. 
 As a first step, we compare the transition costs $\tcostname{\VE}$ and $\oldtcostname{\VE}$
 by getting further insight into   the `gap variation' induced by
 the cost $\icostname$, i.e.\  $\Gap{\icostname}\teta E
 = \sum_{I\in \hole E} \icost{\teta(I^-)}{\teta(I^+)} $. In this way, we
 unveil the structure of  the `gap variation' associated with 
  the viscous correction $\delta$
 for  the crack propagation model and relate the transition costs $\oldtcostname{\VE}$ 
 and $\tcostname{\VE}$.

 {\GGG First of all, by Lemma \ref{l:GapVar1} we immediately get}
 \begin{lemma}
 \label{l:GapVar}
 Let  $E \subset \R$ be  compact subset and  $\teta \in  \Ctran(E;\Xamb)$. Then,
 \begin{align}
 \label{Gap-delta-explicit}
 &
 \Gap{\delta}\teta E = \Gap{\ATWname}\teta E + \mu   \Vars{\icostname}\teta E\,,
 \\
 &
  \label{old-tcost-explicit}
  \oldtcost{\VE}t\teta E = \huno(\teta(E^+){\setminus}\teta(E^-)) + \tcost{\VE}t\teta E \,.
 \end{align}
 \end{lemma}
Relying on Lemma \ref{l:GapVar} and on the previously proved 
Lemma {\GGG \ref{l:K-sigma-regulated}},  we are now in a position to show that the 
energy-dissipation balances \eqref{enbal-VE} and \eqref{old-enbal-VE} do coincide.
\begin{proposition}
\label{prop:they-coincide}
Let $K : [0,T]\to \Xamb$ fulfill  {\GGG
  \eqref{irreversible}}. 
Then,
\begin{equation}
\label{evviva-uguali}
 \huno(K(t){\setminus}K(0)) 
+  \Jvar {\vecostname}{K}{0}{t}  = 
 \Vari{\mdn} K0t 
+  \Jvar {\incostname}{K}{0}{t}  \qquad \text{for all } t \in [0,T]\,
\end{equation}
In particular, $K$ satisfies \eqref{enbal-VE} 
if and only if it fulfills  \eqref{old-enbal-VE}.
\end{proposition}
\begin{proof}
  {\GGG It is not restrictive to check \eqref{evviva-uguali} for
    $t=T$.}
  If $K : [0,T]\to \Xamb$
  fulfills 
  {\GGG \eqref{irreversible} and at least one of the two sides in
  \eqref{evviva-uguali} is finite,
  it is immediate to check that $\Vari{\icostname}K0T<\infty$
  (this property is trivial if the right-hand side of
  \eqref{evviva-uguali} is finite; it follows from \eqref{eq:3} and
  \eqref{only-jump}
  if the left-hand side of
  \eqref{evviva-uguali} is finite).}
Then, Corollary  \ref{cor:3.10} applies, yielding that 
\begin{equation}
\label{1st-ingredient-coincidence}
 \huno(K(T){\setminus}K(0))  =  \Vari{\mdn} K0T  - \lambda  \sum_{s
   \in \jump K }
 \left( \icost{\llim Ks}{K(s)}{+} \icost{K(s)}{\rlim K s} \right)\,.
\end{equation}
{\GGG On the other hand, 
  thanks to \eqref{old-tcost-explicit}, for every $K_-, K_+ \in \Xamb$ with 
$K_-\subset K_+$
we have 
\[
\vecost t{K_-}{K_+} = \oldvecost  t{K_-}{K_+}  - \huno(K^+{\setminus}K_-) = \incost t{K_-}{K_+} + \lambda \icost {K_-}{K_+}\,,
\]
so that 
\begin{equation}
\label{second-ingredient-coincidence}
\Jvar {\vecostname}{K}{0}{T}  = \Jvar{\incostname} K 0T+  \lambda  \sum_{s \in \jump K {\cap} [0,T]} \left( \icost{\llim Ks}{K(s)}{+} \icost{K(s)}{\rlim K s} \right)\,.
\end{equation}
Combining \eqref{1st-ingredient-coincidence} and \eqref{second-ingredient-coincidence}, we deduce  \eqref{evviva-uguali}.}
\end{proof}
\EEE
\subsection{Existence and properties of $\VE$ solutions}
\label{ss:4.2}
This section collects all of our results on $\VE$ solutions for the  system 
$\CRIS$,  
with $\calE$, $\htop$, $\mdn$ and $\delta$ defined in    \eqref{energy}, \eqref{hausd-distance},  \eqref{dissipation-quasidistance}, \eqref{delta}, respectively:
 first and foremost, the existence Theorem \ref{thm:existVEcrack}. 
 \par
  As mentioned in the Introduction, $\VE$ solutions are constructed as follows:  for a given partition $\mathscr{P}_\tau = \{0= t^0_\tau<t^1_\tau<\ldots<t^{N_\tau}_\tau=T\}$ of the interval $[0,T]$ with time step $\tau: = \max_{i=1,\ldots N_\tau} (t_\tau^{i}{-}t_\tau^{i-1})$, 
and  an assigned datum $K_0 \in \K$, 
we consider the minimum problem 
\begin{equation}
\label{TIM}
K^i_\tau \in \Argmin_{K \in \Xs} \big\{\calE(t^i_\tau,K)   + \mathsf{d}(K^{i-1}_\tau,K) +   \delta(K^{i-1}_\tau, K)  \big\} \qquad \text{for } i=1,\ldots, N_\tau,
\end{equation}
which admits a solution thanks to the previously proved lower semicontinuity properties of $\mathsf{d}$ and $\delta$, and the lower semicontinuity/coercivity properties
of $\calE$ that will be precisely stated in Section \ref{ss:5.1-old}.  We introduce the (left-continuous) piecewise constant interpolant  of the elements $(K_\tau^i)_{i=1}^{N_\tau}$
\begin{equation}
\label{pcw-const}
\pwc K\tau \colon [0,T] \to \Xamb \qquad \pwc K\tau(0): = K_0, \qquad \pwc K\tau(t) := K^i_\tau\quad \text{if } t\in (t_\tau^{i-1}, t_\tau^i]\,.
\end{equation}
We are now in a position to give our existence result, stating the convergence of the above interpolants to a $\VE$ solution.
  Let us mention in advance that, starting from an initial datum $K_0 \in \Khf$ for some $h\geq 1$, we   construct a fracture evolution with values in some $\Kmf$, also providing 
 an explicit bound on the index $m$,  cf.\ \eqref{bound-connected-components} below. 

\begin{theorem}[Existence of $\VE$ solutions]
\label{thm:existVEcrack}
Assume that the time-dependent Dirichlet loading fulfills
\begin{equation}
\label{Dirichlet-load}
g\in \mathrm{C}^1([0,T];H^1(\Omega)). 
\end{equation}
Let $K_0 \in \Khf$ for some $h\geq 1$. 
Then, there exists a visco-energetic solution $K$  of the  system $\CRIS$ for brittle fracture  with  such that $K(0)=K_0$. \REVISAGAIN Indeed,  \EEE
for every sequence $(\tau_k)_k$ of time steps with $\tau_k \down 0$ as $k\to\infty$ there exist
  a (not relabeled) subsequence  of $\pwc K{\tau_k}$  
   and a $\VE$ solution $K$ such that  
\begin{equation}
\label{converg-interp}
\pwc K{\tau_k}(t) \wsigma K(t) \qquad \text{for all } t \in [0,T].
\end{equation}
Finally,
 every $\VE$ solution $K$ with 
$K(0)=K_0$ satisfies for every $t\in [0,T]$ 
\begin{equation}
\label{bound-connected-components}
\text{$K(t) \in \Kmf$, 
 with  } 
m \leq h + \REVISAGAIN \tfrac1{\lambda{+}\mu} \EEE  \exp(C_P T) (\ene 0{K_0}+1),
\end{equation} 
where $C_P$ is  the constant defined in \eqref{constant-CP} ahead.
\end{theorem}
\par
The \emph{proof} of Theorem \ref{thm:existVEcrack} will be carried out in Section \ref{s:5} based on  some preliminary results in 
which 
 we are going to show  
that the dissipation distance $\mdn$ defined \eqref{dissipation-quasidistance}, the viscous correction $\corrn$ in \eqref{delta}, and the driving energy functional
$\calE$ in \eqref{energy}  satisfy  a series of  \REVISDOUBT properties \EEE  that are at the heart  of the general existence result \cite[Thm.\ 3.9]{SavMin16} \REVISDOUBT for $\VE$ solutions. Such properties \EEE are 
conditions $\mathbf{<A>}$, $\mathbf{<B>}$, and $\mathbf{<C>}$ stated at the beginning of Section \ref{s:5}.
Relying on their \REVISDOUBT validity and on the fact that our $\VE$ solutions for the crack propagation model 
are indeed $\VE$ solutions in the sense of \cite{SavMin16} (cf.\ Proposition \ref{prop:they-coincide}), \EEE
we will  deduce the proof of   Theorem \ref{thm:existVEcrack} from 
\cite[Thm.\ 3.9]{SavMin16}. 

\par
In \cite{SavMin16} several results on the characterization of the $\VE$ concept, and
 on \emph{optimal jump transitions}, were proved. 
 As we will see in Section  \ref{s:5} ,
 such results also hold for our specific rate-independent system for brittle fracture,
 cf.\ Propositions \ref{prop:charact-VE}  and \ref{prop:OJT-VE} below.
 \par 
 Proposition\  \ref{prop:charact-VE}  provides 
 a twofold characterization of visco-energetic solutions.  First of all, in analogy to the properties 
of energetic and balanced viscosity solutions, 
for  a curve $K\colon  [0,T]\to \Xamb$ that is stable in the visco-energetic sense, the validity of the energy balance \eqref{enbal-VE} is equivalent to the corresponding  energy inequality $\leq$ (cf.\ \eqref{enue-VE}).  It is also equivalent to the validity of an energy-dissipation  inequality that solely involves the dissipation distance $\mdn$, 
cf.\ \eqref{enue-EN} below,
joint with jump conditions that also feature the $\VE$ cost~$\vecostname$. 
As we have recalled in the Introduction, the notion of quasistatic evolution in brittle fracture  features \eqref{enue-EN}, joint with a $\mdn$-stability condition. 
Therefore, the characterization provided by  Proposition \ref{prop:charact-VE}(2) highlights that 
 $\VE$ solutions  essentially differ from  quasistatic evolutions in the description of the energetic behavior of the system at jumps.  \EEE
\begin{proposition}{\cite[Prop.\ 3.8]{SavMin16}}
\label{prop:charact-VE}
Let the assumptions of Theorem\ \ref{thm:existVEcrack} hold.
Let $K\colon  [0,T]\to \Xamb$ satisfy the $\cmdn$-stability condition \eqref{stab-VE}.
\REVISAGAIN Then, the 
following conditions are equivalent:\EEE
\begin{enumerate}
\item $K$ satisfies the $(\huno,\vecostname)$-energy-dissipation balance \eqref{enbal-VE}{\rm;}
\item   $K$ satisfies the $(\huno,\vecostname)$-energy-dissipation upper estimate
\begin{equation}
\label{enue-VE}
\REVISDOUBT \ene T{K(T)} +\huno(K(T){\setminus}K(0))  
+  \Jvar {\vecostname}{K}{0}{T} \leq \ene 0{K(0)}+\int_0^T \partial_t \ene s{K(s)} \dd s ;  \EEE
\end{equation}
\item  $K$ satisfies  the $\huno$-energy-dissipation upper estimate for every $[s,t]\subset [0,T]$
\begin{equation}
\label{enue-EN}
\ene t{K(t)} +\huno(K(t){\setminus}K(s))
\leq \ene s{K(s)} +\int_s^t \partial_t \ene r{K(r)} \dd r, 
\end{equation}
joint with the following jump conditions at every jump point $t\in \jump K$:
\begin{equation}
\label{jump-conditions}
\begin{aligned}
&
\ene t{\lli Kt} - \ene t{K(t)}  && = &&  
\REVISDOUBT \calH^1(K(t){\setminus}\lli Kt) 
+ \vecost t{\lli Kt}{K(t)}, \EEE 
\\
& 
\ene t{K(t)} - \ene t{\rli Kt}  && =  &&  
\REVISDOUBT \calH^1(\rli Kt{\setminus}K(t)) 
+ \vecost t{K(t)}{\rli Kt}\,, \EEE
\\
& 
 \ene t{\lli Kt} - \ene t{\rli Kt}  && =  &&  
\REVISDOUBT  \calH^1(\rli Kt{\setminus}\lli Kt) 
+ \vecost t{\lli Kt}{\rli Kt}\,.  \EEE
\end{aligned}
\end{equation}
\end{enumerate}
\end{proposition}
\noindent 
\REVISAGAIN In fact, cf. \cite[Prop.\ 3.8]{SavMin16} the $(\huno,\vecostname)$-energy-dissipation balance  is equivalent to \eqref{enue-EN}, joint with the jump inequalities $\geq$; nonetheless, for clarity we have preferred to give the jump conditions in the stronger form \eqref{jump-conditions}.  \EEE
\par
Finally, let us  gain further insight into the description of the system behavior at jumps provided by the $\VE$ concept,
via the  properties of 
\emph{optimal jump transitions}. 
Given $t\in [0,T]$ and $K_-, \, K_+\in \Xamb$, 
an admissible transition curve  $\teta \in \Ctran(E;\Xamb)$, with $E\Subset \R$, 
is an optimal transition between $K_-$ and $K_+$ at time $t\in [0,T]$ if it is  a minimizer for 
$\vecost t{K_-}{K_+}$, namely
\begin{equation}
\label{def:OJT}
\teta(E^-) = K_-, \quad \teta(E^+) = K_+,  \quad   \tcost{\VE}{t}{\teta}{E} = \vecost t{K_-}{K_+}\,.
\end{equation}
Furthermore,  we say that $\teta$ is  a
\begin{compactenum} 
\item \emph{sliding transition}, if $\rstab t{\teta(s)} =0$ for  all $s\in E$;
\item \emph{viscous transition}, if  $\rstab t{\teta(s)} >0$ for  all $s\in E \setminus \{ E^-,E^+\}$.
\end{compactenum} 
We have the following result,   cf.\  \cite[Thm.\ 3.14, Rmk.\ 3.15, Cor.\ 3.17, Prop.\ 3.18]{SavMin16}. 
\begin{proposition}
\label{prop:OJT-VE} 
  Let $K: [0,T]\to \Xamb$ be a $\VE$ solution of the  system $\CRIS$ 
  for brittle fracture. 
Then,
\begin{enumerate}
\item   at every jump point $t\in \jump K$ there exists 
 an optimal jump transition $\teta$ between $\lli Kt$ and $\rli Kt$ such that $\teta(s) = K(t)$ for some $s\in E$;
\item
 for a viscous transition $\teta$  between $\lli Kt $ and $\rli K t$
 the  set $E \setminus \{ E^-,E^+\}$ is discrete, i.e.,  all of its points are isolated: namely, $\teta$ is a \emph{pure jump} transition. In fact, $\teta$ may be represented as a finite, or countable, sequence $(\teta_n)_{n\in O}$, with $O $ a compact interval of $\mathbb{Z} \cup \{ \pm \infty\}$, satisfying
 \begin{equation}
 \label{minimum-jump}
 \teta_n \in M(t,\teta_{n-1}) = \mathrm{Argmin}_{K'\in \Xamb}  \left( \calE(t,K') {+}  \cmd{\teta_{n-1}}{K'} \right) 
  \quad \text{for all } n \in O\setminus\{O^-\};
 \end{equation}
  \item
 any optimal jump transition can be canonically decomposed into   an (at most) countable collection  of sliding and viscous transitions.
 \end{enumerate} 
\end{proposition}

\section{Proofs of the main results}
\label{s:5} 
As previously mentioned, prior to carrying out the proof of  Theorem\ \ref{thm:existVEcrack},  in Sections
\ref{ss:5.1-old} and 
 \ref{ss:5.1} ahead
we shall check that the system $\CRIS$ 
given by \eqref{dissipation-quasidistance}, \eqref{delta}, and   \eqref{energy}
complies with a series of 
conditions 
that  were proposed in \cite[Sec.\ 2.2, Sec.\ 3.1, Sec.\ 3.3]{SavMin16}
as a basis for the existence of $\VE$ solutions. 
 Such conditions will also involve the perturbed functional $\calF\colon  [0,T]\times \Xamb \to [0,+\infty]$
\begin{equation}
\label{perturbed-energy}
\calF(t,K):= \ene tK + \md{K_o}{K}
\end{equation}
with $K_o \in \Khf$  for some $h\geq 1$,  a \emph{given}  \REVIS reference crack set. \EEE
\REVISDOUBT Indeed, any $K_o$ contained in the initial crack set $K_0$ may be chosen; for convenience, hereafter we will choose $K_o = \emptyset$, so that $\calF$ reduces to 
$\calF(t,K) = \ene tK + \mathcal{H}^1(K)\GGG +\lambda \icost \emptyset
K$. \EEE
By a sublevel of $\calF$ we mean a set of the form
\[
\{ (t,K) \in [0,T]\times \Kf\, : \ \calF(t,K) \leq r\}
\]
for some $r>0$.
The abstract conditions from 
\cite{SavMin16} read as follows:  
\begin{itemize}
\item[$\mathbf{<A>}$:] 
the energy functional  
$\calE\colon  [0,T]\times \Xamb \to [0,+\infty)$  
is  lower semicontinuous w.r.t.\ the product topology $\htop_\R$  on  the sublevels  of $\calF$, which are  $\htop_\R$-compact;  
 at \emph{every} $(t,K) \in [0,T]\times \Xs  $ there exists $\partial_t \ene tK$; $\partial_t  \calE\colon  [0,T]\times \Xamb \to \R$ is upper semicontinuous  w.r.t.\  $\htop_\R$  on the sublevels of $\calF$,  
and 
\begin{equation}
\label{A} 
\begin{aligned} 
& 
\exists\, C_P>0 \ \ \forall\, (t,K) \in [0,T]\times \Xs \, : \qquad |\partial_t \calE(t,K)| \leq C_P(\ene tK {+}1)\,.
\end{aligned}
\end{equation}
\item [$\mathbf{<B>}$:] the viscous correction $\delta $ is left-$\mdn$-continuous, namely for all sequences $(K_n)_n,\, K \in   \Xs$
\begin{equation}
\label{B-1}
\Big( K_n\wsigma K \text{ and } \md{K_n}{K} \to 0 \text{ as $n\to\infty$ }  \Big) \ \Rightarrow \ \lim_{n\to\infty}\delta(K_n,K) =0
\end{equation}
and for every $(t,K) \in  \mathscr{S}_{\mathsf{D}}$ there holds 
\begin{equation}
\label{B-2}
\limsup_{(s,H) \widetilde{\rightharpoonup} (t,K)}
\frac{\ene s{H} - \ene sK}{\md{H}{K}} \leq 1\,,
\end{equation}
where we have used the place-holder
\[
(s,H) \widetilde{\rightharpoonup} (t,K)  \text{  for  } (  s \to t, \ 
H \wsigma K,  \  \md{H}{K}\to 0,  \ 
	(s,H) \in \mathscr{S}_{\mathsf{D}}, \ s\leq t ). 
	\] 
\item [$\mathbf{<C>}$] For every $Q\geq 0$ the $(\cmdn,Q)$-quasistable sets $\mathscr{S}_{\cmdn}^Q$ have 
 $\htop_\R$-closed  intersections with the sublevels of the functional
$\calF$. 
\end{itemize}
  As observed in \cite{SavMin16}, \EEE  \eqref{B-2} in particular guarantees  that $\cmdn$-stability yields local $\mdn$-stability. 
\par
 Relying on the validity of properties  $\mathbf{<A>}$, $\mathbf{<B>}$, and   $\mathbf{<C>}$, 
in Section \ref{s:proof} ahead 
we shall
conclude the \REVISDOUBT proof of Theorem \ref{thm:existVEcrack}.  
Likewise, also Propositions  \ref{prop:charact-VE} and \ref{prop:OJT-VE} follow, as consequences of  \cite[Prop.\ 3.8, Thm.\ 3.14, Rmk.\ 3.15, Cor.\ 3.17, Prop.\ 3.18]{SavMin16}. \EEE
\subsection{Verification of  properties $\mathbf{<A>}$, $\mathbf{<B>}$, and $\mathbf{<C>}$}
\label{ss:5.1-old} 
 Propositions \ref{l:ad-A}, \ref{l:ad-B}, and \ref{l:ad-C} ahead state the validity of 
properties $\mathbf{<A>}$, $\mathbf{<B>}$, and $\mathbf{<C>}$, respectively, for our system 
$\CRIS$ for brittle fracture. 
Throughout the proof of Propositions \ref{l:ad-A} and \ref{l:ad-C}, we will repeatedly use 
 that, for sequences $(t_n,K_n)_n$ in the sublevels of the functional $\calF$ defined in  \eqref{perturbed-energy}, there holds
\begin{equation}
\label{sole-consequence}
 \sup_n \huno(K_n)<+\infty \quad \text{  and } \quad \exists\, m \geq 1\, : \, (K_n)_n \subset \Kmf 
\end{equation}
as a consequence of Lemma \ref{l:key}. 
\begin{proposition}
\label{l:ad-A}
Under the assumptions of Theorem\  \ref{thm:existVEcrack},
the functional $\calE\colon  [0,T]\times \Xamb \to [0,+\infty)$  defined in  \eqref{energy}
\REVIS and $\partial_t \calE$ from \eqref{power} \EEE
are  continuous  w.r.t.\ the $\htop_\R$-topology  on the sublevels of $\calF$ 
 and $\partial_t \calE$  fulfills \eqref{A}. 
\end{proposition}
\begin{proof} 
 Let $(t_n,K_n)_n\subset [0,T]\times \Xamb$ with 
$\sup_n \calF(t_n,K_n)<+\infty$ converge to some $(t,K)$ w.r.t.\ the $\htop_\R$-topology.
It follows from \eqref{sole-consequence} and 
Theorem \ref{Golab2}
that $K\in \Kmf$, too. 
Since $g \in \mathrm{C}^1([0,T]; H^1(\Omega))$, we have that
$g(t_n) \to g(t)$ in $H^1(\Omega)$. Therefore,  thanks to \eqref{sole-consequence}
we may apply Proposition \ref{prop:continuity-minimizers} and conclude that any sequence
 $(u_n)_n $ with $u_n \in \Argmin_{v \in \Adm{g(t_n)}{K_n}}  \int_{\Omega{\setminus}K_n} \tfrac12 |\nabla v|^2 \dd x$  fulfills 
 $\nabla u_n \to \nabla u$ as $n\to\infty$ in $L^2(\Omega;\R^2)$, with  $u \in \Argmin_{v \in \Adm{g(t)}{K}} \int_{\Omega{\setminus}K} \tfrac12 |\nabla v|^2 \dd x$.
 Then,
 \[
 \ene {t_n}{K_n} = \int_{\Omega{\setminus}K_n} \tfrac12 |\nabla u_n|^2 \dd x \to 
 \int_{\Omega{\setminus}K} \tfrac12 |\nabla u|^2 \dd x = \ene tK \qquad \text{as } n \to\infty.
 \]
 \par
 Since $g \in \mathrm{C}^1([0,T];H^1(\Omega))$, formula \eqref{power} gives $\partial_t \ene tK$ at all $(t,K) \in [0,T]\times \K$. We have 
 \[
 |\partial_t \ene tK| \leq \int_{\Omega{\setminus}K}|\nabla \dot{g}(t)| |\nabla u| \dd x \leq \| \nabla \dot{g}\|_{L^\infty (0,T;L^2(\Omega;\R^2))} \left( \int_{\Omega{\setminus}K} \tfrac12 |\nabla u|^2 \dd x + \tfrac12 \calL^2(\Omega{\setminus}K) \right).
 \]
Then, estimate \eqref{A} follows with 
\begin{equation}
\label{constant-CP}
C_P  =  \| \nabla \dot{g}\|_{L^\infty (0,T;L^2(\Omega;\R^2))} \left(\tfrac12 \calL^2(\Omega) {\vee} 1\right) 
\end{equation}
 The very same arguments  used for the continuity of $\calE$, combined with the fact that $\dot{g} \in \mathrm{C}^0([0,T];H^1(\Omega))$, in fact yield the $\htop_\R$-continuity of $\partial_t \calE$.   This concludes the proof. 
\end{proof}
\par
 With the following result we check the validity of 
 condition $\mathbf{< B>}$;  we shall in fact prove the stronger condition  \eqref{strongerB} below. 
\begin{proposition}
\label{l:ad-B}
The dissipation distance $\mdn$ defined in \eqref{dissipation-quasidistance}
and the viscous correction $\corrn$ in \eqref{delta} fulfill  
\begin{equation}
\label{strongerB} 
\lim_{n\to\infty} \frac{\delta(K_n,K)}{\md{K_n}K} =0 \qquad \text{for all $(K_n)_n, \ K \in \Xamb$ such that }  K_n \wsigma K \text{ and }  \lim_{n\to\infty} \md{K_n}K=0.
\end{equation}
In particular, conditions \eqref{B-1} and \eqref{B-2} are satisfied.
\end{proposition}

\begin{proof}
Since $ \md{K_n}K\to 0$ as $n\to\infty$, we have that, for $n$ sufficiently large,   $K_n \subset K$
 and the integers $ \icost{K_n}{K}$ are  $0$. 
 Therefore,
it is sufficient to observe that
\begin{equation}
\label{eq:ad-C}
\begin{aligned}
 \frac{\delta(K_n,K)}{\md{K_n}K}  & \leq
\frac1{\calH^1(K{\setminus}K_n)} \int_{K{\setminus}K_n} \dist x{K_n} \dd \calH^1(x)    \leq \mathsf{h}(K_n,K) \longrightarrow 0  \EEE\quad \text{as } n \to\infty, 
\end{aligned}
\end{equation}
\EEE
where the last inequality follows  by the definition of  Hausdorff distance.
\REVIS Then, if   $s \to t$ with $s\leq t$, 
$H \wsigma K $ with $\md{H}{K}\to 0 $, and 
	$(s,H) \in \mathscr{S}_{\mathsf{D}}$, we have that 
	\[
	\limsup_{(s,H) \widetilde{\rightharpoonup} (t,K)}
\frac{\ene s{H} - \ene sK}{\md{H}{K}} \leq \limsup_{H \wsigma K, \ \md{H}{K}\to 0}\frac{\md{H}K+\delta(H,K)}{\md{H}K} = 1\,,
	\] 
	which gives \eqref{B-2}. \EEE
\end{proof}

We conclude this section with a discussion on the closedness of the intersection of the $Q$-stable sets  with  the sublevels of the functional $\mathcal{F}$ introduced in  \eqref{perturbed-energy}. \EEE
 It is immediate to see that this property is guaranteed by the following condition: given 
 a
sequence  $(t_n,K_n)_n \subset \mathscr{S}_{\cmdn}^Q$, 
for some $Q\geq 0$, 
such that  $(t_n,K_n) \stackrel{\htop_\R}{\to} (t,K)$ as $n\to\infty$ 
and $\sup_n \calF(t_n,K_n)<+\infty$, 
for every  $K'\in \Xamb$,  with $K'\supset K$  and 
$\md{K}{K'}<+\infty$,  
 we can exhibit  a sequence $(K_n')_n$ such that $K_n'\supset K_n$ and 
\begin{equation}
\label{rec-seq}
\begin{aligned}
&
\limsup_{n\to\infty} \left(\ene{t_n}{K_n'} {-} \ene{t_n}{K_n} {+}\md{K_n}{K_n'} {+}\delta(K_n,K_n') {+}Q  \right)
\\
 & \leq \ene{t}{K'} - \ene{t}{K} +\md{K}{K'} +\delta(K,K') +Q \,;
 \end{aligned}
\end{equation}
\REVIS along the footsteps of \cite{MRS06} (see also \cite[Chap.\ 2.4]{MieRouBOOK}), we shall refer to  $(K_n')_n$ as a `mutual recovery sequence'. \EEE
In this way, we obtain $\ene{t}{K'} - \ene{t}{K} +\md{K}{K'} +\delta(K,K') +Q \geq 0$ for all $K'\in \Xamb$,  
whence $(t,K)\in \mathscr{S}_\cmdn^Q$. 
Indeed, in Proposition \ref{l:ad-C} below we shall obtain \eqref{rec-seq} in a  stronger form.
\begin{proposition}
\label{l:ad-C}
Let $(t_n,K_n)_n \subset \mathscr{S}_{\cmdn}^Q$ be a sequence of  $Q$-stable points 
fulfilling   $\sup_n \calF(t_n,K_n)<+\infty$. 
Suppose that  $(t_n,K_n) \stackrel{\htop_\R}{\to} (t,K)$.  
Then, for every  $K' \in \Xs$  with $K'\supset K$ 
and $\md{K}{K'}<+\infty$
there exists a sequence $(K_n')_n$ such that  $K_n'\supset K_n$ and  the following convergences hold as $n\to\infty$: 
\begin{subequations}
\label{rec-seq-better}
\begin{align}
&
\label{rec-seq-0}
K_n'\wsigma K',
\\
&
\label{rec-seq-1}
\ene{t_n}{K_n'} \to \ene t{K'},
\\
&
\label{rec-seq-2}
\md{K_n}{K_n'}  \to \md{K}{K'},
\\
&
\label{rec-seq-3}
 \delta(K_n,K_n')  \to \delta(K,K').
\end{align}
\end{subequations}
In particular, condition $\mathbf{<C>}$ is valid. 
\end{proposition}
The \emph{proof} shall be carried out in the upcoming Section  \ref{ss:5.1}. 
\subsection{Proof of Proposition \ref{l:ad-C}}
\label{ss:5.1} 
 Since $\sup_n \calF(t_n,K_n)<+\infty$, we have that $(K_n)\subset \Kmf$ for some $m\geq 1$. 
Along the footsteps of \cite{DMToa02},   first of all we shall 
 exhibit a \REVIS mutual recovery sequence \EEE for a fixed competitor set  $K'=\trial $ 
  that is, additionally, connected, i.e.\ $\trial \in \tKf$, cf.\ the upcoming Lemma \ref{l:new3.8}. Then, in Lemma  \ref{l:new3.5} we will address the general case  in which the competitor set  is in $\Kpf$
  for some $p\geq1$. The proof of Proposition \ref{l:ad-C} will be then carried out at the end of this section.  \EEE
   The proofs of   Lemmas   \ref{l:new3.8}  and  \ref{l:new3.5}  strongly rely on the arguments for \cite[Lemmas 3.8 \& 3.5]{DMToa02}. 
\begin{lemma}
\label{l:new3.8}
Let $m \in \N$, let $(K_n)_n,\, K \in  \tKmf$ 
fulfill $\mathsf{h}(K_n,K) \to 0$ as $n\to\infty$, and let $\trial \in \tKf$ with $\trial \supset K$. Then, there exists a sequence $(\trial_n)_n\subset \tKf$ such that 
$\trial_n \supset K_n$ and 
\begin{subequations}
\label{props-trialn}
\begin{align}
\label{props-trialn-1}
&
\mathsf{h}(\trial_n,\trial) \to 0  && \text{as } n \to \infty,
\\
&
\label{props-trialn-2}
\calH^1(\trial_n {\setminus} K_n) \to \calH^1(\trial{\setminus} K), 
&& \text{as } n \to \infty,
\\
& 
\label{props-trialn-3} 
\int_{\trial_n {\setminus} K_n} \dist x{K_n} \dd \calH^1(x) \to \int_{\trial {\setminus} K} \dist x{K} \dd \calH^1(x) &&  \text{as } n \to \infty\,.
\end{align}
\end{subequations}
\end{lemma}
\begin{proof}
  
\par
If $K = \emptyset$, it is sufficient to define $\trial_n: = \trial$. Indeed, from $\mathsf{h}(K_n,\emptyset) \to 0$ we deduce that $K_n = \emptyset $ for $n$ sufficiently large,
and then  the convergences properties \eqref{props-trialn} are trivially satisfied.
\par
Let us now assume $K \neq \emptyset$, and let $K^1,\ldots, K^i$, $ 1\leq i \leq m$, be its connected components.
First of all, in Step $1$ we will provide the construction of a \REVIS mutual recovery sequence \EEE with the desired properties \eqref{props-trialn} for a carefully chosen set $\widehat{\trial}$, such that  $\widehat{\trial}$ coincides with $K$, if $K$ is connected,  and $\widehat\trial $ is a suitable  subset of $J$ containing   $K$ (cf.\  \eqref{hat-trial-i}) in the general case. 
\paragraph{\bf Step $1$.}
If $i=1$, let us set 
\begin{subequations}
\label{hat-trial}
\begin{equation}
\label{hat-trial-1}
\widehat{\trial} : = K = K^1.
\end{equation}
 If $ i \geq 2$, we apply  \cite[Lemma 3.7]{DMToa02} to conclude that there exists a finite family of indices $(\sigma_j)_{j=0}^\ell$, with $\{ \sigma_0, \ldots, \sigma_\ell \} = \{ 1, \ldots, i\}$, and a family $(\Gamma_j)_{j=1}^\ell$ of connected components of $\trial \setminus K$, such that 
$K^{\sigma_{j-1}} \cap \overline{\Gamma}_j \neq \emptyset \neq K^{\sigma_{j}} \cap \overline{\Gamma}_j $ for $j=1,\ldots, \ell$, namely $ \overline{\Gamma}_j$ connects 
$K^{\sigma_{j-1}} $ to $K^{\sigma_{j}}.$ In this case, we set
\begin{equation}
\label{hat-trial-i}
\widehat{\trial} : = K  \cup \bigcup_{j=1}^\ell  \overline{\Gamma}_j
\end{equation}
\end{subequations}
 and prove the following

\noindent
{\sl \textbf{Claim}: 
there exists a sequence $(\widehat{\trial}_n)_n \subset \tKf$ such that $\widehat{\trial}_n \supset K_n$ and } 
\begin{subequations}
\label{props-hat-trialn}
\begin{align}
\label{props-hat-trialn-1}
&
\mathsf{h}(\widehat\trial_n,\widehat\trial) \to 0  && \text{as } n \to \infty,
\\
&
\label{props-hat-trialn-2}
\calH^1(\widehat\trial_n {\setminus} K_n) \to \calH^1(\widehat\trial{\setminus} K) 
&& \text{as } n \to \infty,
\\
& 
\label{props-hat-trialn-3} 
\int_{\widehat\trial_n {\setminus} K_n} \dist x{K_n} \dd \calH^1(x) \to \int_{\widehat\trial {\setminus} K} \dist x{K} \dd \calH^1(x) && \text{as } n \to \infty\,.
\end{align}
\end{subequations}

To carry out the construction of the sets $\widehat{\trial}_n$, we proceed in the following way. Given the connected components $(K^l)_{l=1}^i$ of $K$, we choose $\eps>0$ such that the sets $\{ x \in \overline\Omega\, : \ \dist x {K^l} \leq \eps\}$  are pairwise disjoint, and we set
\[
\widetilde{K}_n^l : = \{ x \in K_n\, : \ \dist x{K^l} \leq \eps\}.
\]
Following \cite{DMToa02}, we observe that, for sufficiently large $n$,  we have that
$K_n=\widetilde{K}_n^1  \cup \ldots \cup \widetilde{K}_n^i $, 
$\widetilde{K}_n^l  \in \tKmf$, and 
$\mathsf{h}(\widetilde{K}_n^l, K^l) \to 0$ as $n\to\infty$ for all $l \in \{1,\ldots, i\}$.
We now apply  \cite[Lemma 3.6]{DMToa02}
and for all $l \in \{1, \ldots, i\}$ we  find a sequence $(\widehat{K}_n^l)_n \subset \tKf$ such that $\widehat{K}_n^l \supset \widetilde{K}_n^l$,
\begin{subequations}
\label{tilde2hat}
\begin{equation}
\label{tilde2hat-1}
\mathsf{h}(\widehat{K}_n^l, K^l) \to 0, \quad \text{ and } \quad \calH^1 (\widehat{K}_n^l{\setminus} \widetilde{K}_n^l) \to 0 \qquad \text{as } n \to \infty. 
\end{equation}
Therefore, 
$\mathsf{h}(\widehat{K}_n^l, \widetilde{K}_n^l) \leq \mathsf{h}(\widehat{K}_n^l, K^l) + 
\mathsf{h}(\widetilde{K}_n^l, K^l) \to 0$, as $n\to\infty$. This implies that 
\begin{equation}
\label{tilde2hat-2}
\int_{\widehat{K}_n^l{\setminus} \widetilde{K}_n^l} \dist x{ \widetilde{K}_n^l} \dd \calH^1(x) \to 0  \qquad \text{as } n \to \infty. 
\end{equation}
\end{subequations}
\par
In the case $i=1$ (namely, $K = K^1 \in \tKf$), we define 
$
\widehat\trial_n : = \widehat{K}_n^1 \in \tKf$.  Then, properties \eqref{props-hat-trialn} are satisfied: indeed, 
 in this case
$\widehat\trial = K$, so that the first of \eqref{tilde2hat-1} yields \eqref{props-hat-trialn-1}. 
Furthermore, 
$\widetilde{K}_n^1 : = \{ x \in K_n\, : \ \dist x{K^1}  =  \dist x{K}  \leq \eps\} $ coincides with $K_n$ for $n$ large enough.
Therefore,
$\widehat\trial_n \supset K_n$ and
$\calH^1(\widehat\trial_n{\setminus}K_n) = \calH^1(\widehat{K}_n^1{\setminus}\widetilde{K}_n^1) \to 0$ by the second of \eqref{tilde2hat-1}. 
 Property \eqref{props-hat-trialn-3} then follows from  \eqref{tilde2hat-2}, as $\widehat\trial\setminus K = \emptyset$.  
\par
Suppose now that $K$ is not connected, namely $i\geq 2$. Then, the set $\widehat\trial$ is given by \eqref{hat-trial-i}. For every $j=1,\ldots, \ell$, we fix $x^j \in K^{\sigma_{j-1}} \cap \overline{\Gamma}_j $ and $y^j \in K^{\sigma_{j}} \cap \overline{\Gamma}_j $. Since $\mathsf{h}(\widehat{K}_n^l, K^l) \to 0$ as $n\to\infty$ for all $l \in \{1,\ldots, i\}$, we have that there exist sequences $(x_n^j)_n,\, (y_n^j)_n$ with $x_n^j \in \widehat{K}_n^{\sigma_{j-1}}$ and $y_n^j \in \widehat{K}_n^{\sigma_{j}}$ for all $n\in \N$, such that 
$x_n^j \to x^j$ and $y_n^j \to y^j$ as $n\to\infty$. Since $\Omega$ has a Lipschitz boundary, there exist arcs $X_n^j$ and $Y_n^j$ in $\overline\Omega$, connecting $x_n^j$ to $x^j$ and 
$y_n^j$ to $y^j$, respectively, such that 
$\calH^1(X_n^j) \to 0$ and $\calH^1(Y_n^j)\to 0$ as $n\to\infty$.
We set for every $n\in \N$
\begin{equation}
\label{def-wht-notconn}
\widehat\trial_n : = \bigcup_{l=1}^i \widehat{K}_n^l  \cup  \bigcup_{j=1}^\ell X_n^j   \cup  \bigcup_{j=1}^\ell \overline{\Gamma}_j 
 \cup  \bigcup_{j=1}^\ell Y_n^j \,.
\end{equation}
It has been shown in the proof of \cite[Lemma  3.8]{DMToa02} that $\widehat\trial_n \in \tKf$ for sufficiently large  $n$, and that  \eqref{props-hat-trialn-1} and  
 \eqref{props-hat-trialn-2} hold. 
   It remains to check \eqref{props-hat-trialn-3}. With this aim, we observe that by \eqref{def-wht-notconn} we have 
\[
\begin{aligned}
\int_{\widehat\trial_n {\setminus} K_n} \dist x{K_n} \dd \calH^1(x)  &  = \int_{\widehat\trial_n} \dist x{K_n} \dd \calH^1(x)  \\ &  = 
\sum_{l=1}^i \int_{\widehat{K}_n^l} \dist x{K_n} \dd \calH^1(x) + \sum_{j=1}^\ell \int_{X_n^j} \dist x{K_n} \dd \calH^1(x)
\\ & \quad 
+ \sum_{j=1}^\ell \int_{\Gamma_j} \dist x{K_n} \dd \calH^1(x)
+ \sum_{j=1}^\ell \int_{Y_n^j} \dist x{K_n} \dd \calH^1(x) 
\\ & \doteq S_n^1+S_n^2+S_n^3+S_n^4
\end{aligned}
\]
(where the integrals are taken over $\Gamma_j $ since  $\calH^1(\overline{\Gamma}_j)
=\calH^1(\Gamma_j)$ by \cite[Prop.\ 2.5]{DMToa02}).
As for the first summand, observe that  
\[
\int_{\widehat{K}_n^l} \dist x{K_n} \dd \calH^1(x)  \leq  
\int_{\widehat{K}_n^l} \dist x{\widetilde{K}_n^l} \dd \calH^1(x) \longrightarrow  0 \text{ as } n \to\infty  \quad \text{for every $l =1,\ldots, i$,}
\]
 where the inequality is due to the fact that $K_n=\widetilde{K}_n^1  \cup \ldots \cup \widetilde{K}_n^i \supset \widetilde{K}_n^l$,  while the convergence to $0$ is proved in  \eqref{tilde2hat-2}. Therefore,  $S_n^1 \to 0$ as $n\to\infty$.
 We trivially estimate
 \[
 \int_{X_n^j} \dist x{K_n} \dd \calH^1(x) \leq \mathrm{diam}(\Omega) \cdot \calH^1(X_n^j) \to 0 \text{ as } n \to\infty  \quad \text{for all } j=1,\ldots, \ell,
 \]
and we handle the terms $ \int_{Y_n^j} \dist x{K_n} \dd \calH^1(x)$ in the same way.
We thus conclude that 
 $S_n^2 \to 0$  and  $S_n^4 \to 0$ as $n\to\infty$.
 Finally, we observe that
  \[
 \limsup_{n\to\infty} \int_{\Gamma_j} \dist x{K_n} \dd \calH^1(x)
\leq  \int_{\Gamma_j}  \limsup_{n\to\infty} \dist x{K_n} \dd \calH^1(x)
\leq  \int_{\Gamma_j} \dist x{K} \dd \calH^1(x)  \quad \text{for all } j=1,\ldots, \ell,
 \]
where the first inequality follows from the Fatou Lemma, and the second one is a straightforward consequence of  the fact that $K_n\to K$ w.r.t.\ the Hausdorff distance. All in all, we conclude that 
\[
\begin{aligned}
\limsup_{n\to\infty} \int_{\widehat\trial_n {\setminus} K_n} \dist x{K_n} \dd \calH^1(x) 
\leq  \limsup_{n\to\infty} \sum_{j=1}^\ell \int_{\Gamma_j} \dist x{K_n} \dd \calH^1(x)
 & = \sum_{j=1}^\ell  \int_{\Gamma_j} \dist x{K} \dd \calH^1(x) 
 \\ & 
=  \int_{\widehat\trial{\setminus}K} \dist x{K} \dd \calH^1(x),  
\end{aligned}
\]
where the last equality is  due to \eqref{hat-trial-i}.
Then, \eqref{props-hat-trialn-3} ensues, since 
$\liminf_{n\to\infty} \int_{\widehat\trial_n {\setminus} K_n} \dist x{K_n} \dd \calH^1(x) $ is estimated from below by 
$\int_{\widehat\trial{\setminus}K} \dist x{K} \dd \calH^1(x)$ thanks to   Proposition \ref{Golab3}.
\paragraph{\bf Step $2$.} 
Let us now carry out the construction of the  \REVIS mutual recovery sequence \EEE $(\trial_n)_n$ for a given 
$\trial \in \tKf$ with $\trial \supset K$. 
Let $\widehat\trial$ be the set introduced in \eqref{hat-trial}. Observe that $\trial$
is locally connected (see \cite[Lemma 1]{Chambolle-Doveri}), hence the connected components of $\trial \setminus \widehat\trial$ are open in the relative topology of $\trial$. Therefore, 
since $\trial$ is separable,  $\trial \setminus \widehat\trial$ 
has at most countably many connected components  $(C_\ell)_{\ell \in L}$,  with $L$ a finite or an infinite subset of $\N$. 
It follows from the proof of \cite[Lemma 3.7]{DMToa02} that each component $C_\ell$ is open in $\trial$ and satisfies $\overline{C}_\ell \cap K\neq \emptyset$. Let us fix a point $z^\ell \in \overline{C}_\ell \cap K$ for every $\ell\in L$. From $\mathsf{h}(K_n,K) \to 0$ as $n\to\infty$ we deduce that there exists a sequence $ (z_n^\ell)_n$ such that $
z_n^\ell  \in K_n$ for all $n\in \N$ and $z_n^\ell \to z^\ell$ as $n\to\infty$. Since $
\Omega $ is Lipschitz, for every $\ell \in L$ there exists an arc $Z_n^\ell \subset\overline\Omega$ connecting $z_n^\ell$ to $z_\ell$, and such that $\calH^1(Z_n^\ell) \to 0$ as $n\to\infty$.
Finally, along the footsteps of \cite{DMToa02}  we observe that there exists a sequence $(\Lambda_n)_n \subset \N$ such that 
\[
\lim_{n\to\infty} \sum_{\ell=1}^{\Lambda_n} \calH^1(Z_n^\ell)  =0
\]
(in fact, if the set $L$ consists of $1\leq \Lambda<+\infty$ elements, then we take $\Lambda_n = \Lambda$).
\par
We claim that the sequence
\begin{equation}
\label{final-definition}
\trial_n: = \widehat{\trial}_n \cup \bigcup_{\ell=1}^{\Lambda_n} Z_n^\ell \cup \bigcup_{\ell=1}^{\Lambda_n}
\overline{C}_\ell
\end{equation}
complies with \eqref{props-trialn}. In fact, it is sufficient to check \eqref{props-trialn-3}, as \eqref{props-trialn-1} and \eqref{props-trialn-2} have been proved in \cite[Lemma 3.8]{DMToa02}.
With this aim, we observe that 
\[
\begin{aligned}
&
\int_{\trial_n {\setminus} K_n} \dist x{K_n} \dd \calH^1(x)  \\  &  = 
\int_{\widehat{\trial}_n} \dist x{K_n} \dd \calH^1(x) + \sum_{\ell=1}^{\Lambda_n} 
\int_{Z_n^\ell}\dist x{K_n} \dd \calH^1(x)  +  \sum_{\ell=1}^{\Lambda_n} \int_{C_\ell}\dist x{K_n} \dd \calH^1(x)   =: S_5^n+S_6^n+S_7^n\,.
\end{aligned}
\]
It follows from \eqref{props-hat-trialn-3} that 
\[
\limsup_{n\to\infty} 
S_5^n
\leq\int_{\widehat\trial {\setminus} K} \dist x{K} \dd \calH^1(x) \,.
\]
We estimate
\[
S_6^n \leq \mathrm{diam}(\Omega) \sum_{\ell=1}^{\Lambda_n}  \calH^1(Z_n^\ell) \to 0 \qquad \text{as } n \to \infty.
\]
Finally, we observe that 
\[
\begin{aligned}
\limsup_{n\to\infty} S_7^n  = \limsup_{n\to\infty}  \sum_{\ell=1}^{\Lambda_n} \int_{C_\ell}\dist x{K_n} \dd \calH^1(x)  &  \leq \limsup_{n\to\infty} \int_{\bigcup_{\ell \in L}C_\ell}\dist x{K_n} \dd \calH^1(x) \\ &   \leq \int_{\bigcup_{\ell \in L}C_\ell}   \limsup_{n\to\infty} \dist x{K_n} \dd \calH^1(x) \\ & \leq \int_{\bigcup_{\ell \in L}C_\ell}   \dist x{K} \dd \calH^1(x), 
\end{aligned}
\]
again by the Fatou Lemma and the fact that  $\mathsf{h}(K_n,K)\to 0$ as $n\to\infty$. 
All in all,  we conclude that 
\[
\begin{aligned}
\limsup_{n\to\infty} \int_{\trial_n {\setminus} K_n} \dist x{K_n} \dd \calH^1(x)   
 & \leq \int_{\widehat\trial {\setminus} K} \dist x{K} \dd \calH^1(x) +  \int_{\bigcup_{\ell \in L}C_\ell}   \dist x{K} \dd \calH^1(x) 
\\ & = \int_{\trial {\setminus} K} \dist x{K} \dd \calH^1(x) \,,
\end{aligned}
\]
namely, an inequality in \eqref{props-trialn-3}. The converse inequality follows from Proposition \ref{Golab3}.
This concludes the proof \REVIS of Lemma \ref{l:new3.8}. \EEE
\end{proof}
\par
As in \cite[Lemma 3.5]{DMToa02}, we now extend the construction of the \REVIS mutual \EEE recovery  sequence to the case  the `competitor set' $\trial$ has at most $p$ connected components, with $p\geq 1$.
\begin{lemma}
\label{l:new3.5}
Let $m,\, p \in \N, \, m,\, p \geq 1$, let $(K_n)_n,\, K \in  \tKmf$ fulfill $\mathsf{h}(K_n,K) \to 0$ as $n\to\infty$, and let $K' \in \tKpf$ with $K' \supset K$. 
Then, there exists a sequence $(K_n')_n\subset \tKpf$ such that
$K_n' \supset K_n$ and 
properties \eqref{rec-seq-better}  hold. 
\end{lemma}
\begin{proof}
As in the proof of  \cite[Lemma 3.5]{DMToa02},
we consider the connected components $\trial^1,\ldots,\trial^i$, $1\leq i\leq p$, of the set $K'$, we fix $\eps>0$ such that the sets 
$\{ x\in \overline\Omega\, : \ \dist x{\trial^l} \leq \eps\}$, $l \in \{1, \ldots, i \}$, are pairwise disjoint, and we define
\[
\widehat{K}_n^l := \{ x \in K_n\, : \dist x{\trial^l} \leq \eps \}, \qquad l \in \{1, \ldots, i \}\,.
\] 
Following \cite{DMToa02}, we observe that, 
for $n $ large enough, 
the sets $\widehat{K}_n^l$ are in $\tKmf$, 
$K_n = \bigcup_{l=1}^i \widehat{K}_n^l $, and  $\mathsf{h}(\widehat{K}_n^l,K^l) \to 0$ as $n\to\infty$, with $K^l :=  K \cap \trial^l$.
If $K^l = \emptyset$, we set $\trial_n^l \equiv \trial^j$ for all $n \in \N$. If $K^l \neq \emptyset$, 
 we apply Lemma \ref{l:new3.8} to the connected sets $\trial^l$ and to the sequences  $(\widehat{K}_n^l)_n$, $l \in \{1, \ldots, i \}$, and for each $l\in \{1,\ldots, i\}$ we  find a sequence $(\trial_n^l)_n \subset \tKf$ such that 
\begin{equation}
\label{appl-lemma-l}
\begin{aligned}
&
\trial_n^l \supset \widehat{K}_n^l, \quad \mathsf{h}(\trial_n^l,\trial^l ) \to 0, \quad   \calH^1(\trial_n^l{\setminus} \widehat{K}_n^l ) \to \calH^1(\trial^l {\setminus} K^l), \\ &  \limsup_{n\to\infty}\int_{\trial_n^l{\setminus} \widehat{K}_n^l } \dist x{\widehat{K}_n^l} \dd \calH^1(x)  \leq \int_{\trial^l{\setminus} K^l } \dist x{K^l} \dd \calH^1(x)\,.
\end{aligned}
\end{equation}
Note that, for $n$ large enough, the sets $(J_n^l)_{l=1}^i$ are pairwise disjoint. 
\par
 Then, we define the \REVIS mutual \EEE recovery sequence $(K_n')_n$ for the set $K'$ in this way:
 \[
 K_n' : = \trial_n^1 \cup  \ldots \cup\trial_n^i\,.
 \]
By construction $K_n' \supset \bigcup_{l=1}^i   \widehat{K}_n^l  =K_n$ and $\mathsf{h}(K_n',K')\to 0$ as $n\to\infty$,
namely \eqref{rec-seq-0} holds. 
Then, \eqref{rec-seq-1} follows from \eqref{rec-seq-0} and Proposition \ref{l:ad-A}. 
Furthermore,
\begin{equation}
\label{combo-1}
\begin{aligned}
\calH^1(K_n'{\setminus}K_n)  & = \calH^1( \bigcup_{l=1}^i \trial_n^l {\setminus} \bigcup_{j=1}^i   \widehat{K}_n^j ) 
=\calH^1( \bigcup_{l=1}^i  ( \trial_n^l {\setminus}\widehat{K}_n^l ) )\\
&
\leq \sum_{l=1}^i \calH^1(\trial_n^l {\setminus}\widehat{K}_n^l ) \longrightarrow  \sum_{l=1}^i  \calH^1(\trial^l {\setminus} K^l) =   \calH^1(K' \setminus K)\,.
\end{aligned}
\end{equation}
where the second equality follows from the fact that $ \trial_n^l  \setminus \widehat{K}_n^j =\trial_n^l$ for
$l\neq j$, 
 analogously, we have $\trial^l {\setminus} K^j =\trial^l$ for $l \neq j$, which gives the very last equality. 
 \par
Now, we calculate
 $
 \icost{K_n}{K_n'} = \icost{\bigcup_{l=1}^i   \widehat{K}_n^l}{\bigcup_{l=1}^i \trial_n^l},
 $
 namely the number of connected components  $\Lambda $ of $\bigcup_{l=1}^i \trial_n^l$  such that $\Lambda \cap \bigcup_{l=1}^i   \widehat{K}_n^l = \emptyset $. 
Since for $n$ large enough the sets $\trial_n^l$, $l=1, \ldots, i$,  are connected and pairwise disjoint, 
each  $\Lambda$ must coincide with a set  $\trial_n^{\bar l}$, for some $\bar l \in \{1,\ldots, i\}$, that 
fulfills
$\trial_n^{\bar l} \cap  \widehat{K}_n^l = \emptyset$ for every $l \in \{1,\ldots, i\}$. 
 Recall that,
 for $n$ sufficiently large,
 $\trial_n^{\bar l} \cap  \widehat{K}_n^{l} = 
  \trial^{\bar l} \cap K^{l} = \emptyset $ for $l \neq \bar l$,  and that,
$\trial_n^{\bar l} \cap  \widehat{K}_n^{\bar l}  = \emptyset $ if and only if $ \trial^{\bar l} \cap K^{\bar l} = \emptyset $. 
 Then, we easily conclude that 
\[
 \icost{K_n}{K_n'} = \icost{\bigcup_{l=1}^i   K^l}{\bigcup_{l=1}^i \trial^l} = \icost{K}{K'} \qquad \text{for $n$ large enough.} 
\]
 Thus, we infer the validity of property \eqref{rec-seq-2}. 
 \par
With the same arguments we find that 
\[
\begin{aligned}
&
\limsup_{n\to\infty}\int_{K_n'{\setminus} \widehat{K}_n } \dist x{\widehat{K}_n^l} \dd \calH^1(x)  
 = \limsup_{n\to\infty} \sum_{l=1}^i \int_{\trial_n^l{\setminus} \widehat{K}_n^l } \dist x{\widehat{K}_n^l} \dd \calH^1(x)
 \\
  &   
\leq  \sum_{l=1}^i  \int_{\trial^l{\setminus} K^l }
 \dist x{K^l} \dd \calH^1(x)
=
\int_{K'{\setminus} K } \dist x{K} \dd \calH^1(x), 
\end{aligned}
\]
where the last inequality follows from \eqref{appl-lemma-l}. This
 gives an inequality in \eqref{rec-seq-3}; the converse one is again due to Proposition \ref{Golab3}. 
This concludes the proof.
\end{proof}
\begin{proof}[Proof of Proposition \ref{l:ad-C}]
By \eqref{sole-consequence} and by Theorem \ref{Golab2}(i), the set $K$ belongs to $\Kmf$ for some $m$. Since $\lambda\icost{K}{K'} \leq \mathsf{d}(K,K')<+\infty$, we have that $K' \in \Kpf$ with $p = m+ \icost{K}{K'}$. Then, the conclusion follows from Lemma \ref{l:new3.5}. 
\end{proof}

\subsection{Proofs}
\label{s:proof}
In this section we prove  Theorem \ref{thm:existVEcrack}. 
\medskip
\begin{proof}[Proof of Theorem  \ref{thm:existVEcrack}]
  The first and the last statement follow  from the general existence result \cite[Thm.\ 3.9]{SavMin16}, which applies to the system $\CRIS$ for brittle fracture thanks to the validity of conditions  $\mathbf{<A>}$, $\mathbf{<B>}$, and  $\mathbf{<C>}$,
  proved in Propositions \ref{l:ad-A}, \ref{l:ad-B}, and  \ref{l:ad-C}. 
  
  \par
We now prove  the explicit bound \eqref{bound-connected-components} on the number $m$ of connected components. Indeed,  \REVISAGAIN
from the energy-dissipation balance \eqref{enbal-VE}
  we infer that 
\begin{equation}
\label{2Gronw}
\Jvar {\vecostname}{K}{0}{t}
\leq
\ene t{K(t)}  + \huno(K(t){\setminus}K(0)) 
+  \Jvar {\vecostname}{K}{0}{t}
\leq 
 \ene 0{K_0}+\int_0^t  C_P(\ene s{K(s)}{+}1) \dd s,
 \end{equation}
\EEE where the last inequality follows from 
 \eqref{power}. Then, by the Gronwall Lemma we infer that 
 \[
 \ene t{K(t)} \leq \left( \ene 0{K_0}{+}1\right) \exp(C_Pt) -1\qquad \text{for all } t \in [0,T].
 \]
Inserting the above estimate in \eqref{2Gronw} yields
 \[
\REVISAGAIN  (\lambda{+}\mu)\Vari{\icostname} K0t \leq
   \Jvar {\vecostname}{K}{0}{t} \leq 
    \left( \ene 0{K_0}{+}1\right) \exp(C_P T) -1,
 \]
where the first inequality is a consequence of \eqref{eq:3}.  Since
$\alpha(K(0), K(t)) \leq \mathrm{Var}_{\icostname}(K, [0,t]) $, we ultimately find 
\[
\alpha(K(0), K(t))  \leq \frac1{\lambda{+}\mu}  \left( \ene 0{K_0}{+}1\right) \exp(C_P T), 
\]
and  \eqref{bound-connected-components}  follows by  Lemma \ref{l:key}. \EEE
This concludes the proof of Theorem \ref{thm:existVEcrack}. 
\end{proof}

\section{Behavior near the crack tips}
\label{s:6}
In the same spirit of \cite[Sec.\ 8]{DMToa02}, in this section we describe the \emph{singularity} at the crack tips of the displacement $u(t)$ associated with a  $\VE$ solution $K(t)$ 
to the system $\CRIS$. This will be examined in 
 an interval $(\tau_0,\tau_1)$ during which $K$ evolves \emph{continuously} as a function of time. Furthermore, along the footsteps of \cite{DMToa02},  we   confine the discussion to the case in which the  (moving part of the)  crack set consists of a \emph{finite} family of simple arcs, whose endpoints are the moving tips of the crack, as specified in Hypothesis \ref{h:griffith} below. 
 In Theorem \ref{th:griffith} below we will show that the $\VE$ solution $K$ complies with Griffith's criterion for crack growth.
\par
Let us specify the structural condition on the crack  $K\colon [0,T]\to \Xamb$. 
\begin{hypothesis}
\label{h:griffith}
We suppose that $K\colon [0,T]\to \Xamb$ fulfills the following condition  on some  $(\tau_0, \tau_1) \subset [0,T]$: there exists a finite family
 $(\Gamma_i)_{i=1}^p$ of  arcs  contained in $\Omega$ and parameterized by 
arc length by $\rmC^2$  bijective  functions $\gamma_i\colon  [\sigma_i^0,\sigma_i^1] \to  \Gamma_i$ such that 
\begin{equation}
\label{structure-condition}
K(t) = K(\tau_0) \cup \bigcup_{i=1}^p \Gamma_i(\sigma_i(t)) \qquad \text{for all } t \in (\tau_0,\tau_1),
\end{equation}
where, for $i=1,\ldots, p$,  $\sigma_i \colon  [\tau_0,\tau_1] \to [\sigma_i^0,\sigma_i^1]$ are non-decreasing  continuous  functions such that 
$\sigma_i(\tau_0)=\sigma_i^0$ and $\sigma_i^0<\sigma_i(t) <\sigma_i^1$, while
$
\Gamma_i(\sigma) = \{ \gamma_i(s)\, : \ \sigma_i^0 \leq s \leq \sigma\}$.  We also assume that the arcs $(\Gamma_i)_{i=1}^p$ are pairwise disjoint, and that $
\Gamma_i \cap K(t_0) = \{ \gamma_i(\sigma_i^0)\} $ for every $i=1,\ldots, p$. 
\end{hypothesis}
\noindent
Hence, for  $ t\in (\tau_0,\tau_1) $ the fracture grows along the branches $\Gamma_i$, $i=1,\ldots, p$, and the points $\gamma_i(\sigma_i(t))$ are the moving crack  tips. 
The compliance with Griffith's  criterion stated in Theorem \ref{th:griffith} ahead will be expressed in terms of conditions involving the \emph{stress intensity factors}  of the displacements $u(t)$ at the crack tips. We briefly recall 
some preliminary facts about this notion.
\smallskip

\paragraph{\bf Basics on the stress intensity factor.}
\  The   notion of stress intensity factor is based on  the following result.
\begin{proposition}
Let $B\subset \R^2$ be an open ball, and let $\serifgamma\colon  [\serifsigma_0,\serifsigma_1]\to \R^2$ be a simple path of class $\rmC^2$ parameterized by arc length, such that $\serifgamma(\serifsigma_0) \in \partial B$, $\gamma(\serifsigma_1) \in \partial B$, and $\serifgamma(\serifsigma) \in B$ for all $\serifsigma \in (\serifsigma_0,\serifsigma_1)$. In addition, assume that $\serifgamma$ is not tangent to $\partial B$ at $\serifsigma_0$ and $\serifsigma_1$. 
\par
Given $\serifsigma \in (\serifsigma_0,\serifsigma_1)$, let $\serifGamma (\serifsigma):  = \{\serifgamma(\mathsf{s})\, : \ \serifsigma_0 \leq s \leq \serifsigma\}$ and let $\serifu \in L^{1,2}(B{\setminus}\serifGamma (\serifsigma)) $ satisfy
\[
\int_{B{\setminus}\serifGamma (\serifsigma)} \nabla \serifu \cdot \nabla z \dd x =0 \qquad \text{for all } z \in L^{1,2}(B{\setminus}\serifGamma (\serifsigma)) \text{  with } z=0 \text{  on } \partial B {\setminus}\serifGamma (\serifsigma)\,.
\]
Then, there exists a unique constant $\kappa = \kappa (\serifu,\serifsigma) \in \R$ such that
\begin{equation}
\label{definition-kappa}
\serifu -2\kappa\sqrt{\rho/\pi}\sin (\theta/2) \in H^2(B{\setminus}\serifGamma (\serifsigma)) \cap H^{1,\infty} (B{\setminus}\serifGamma (\serifsigma))\,,
\end{equation}
where $\rho(x) = |x-\serifgamma(\serifsigma)|$ and $\theta(x)$ is the continuous function on $B{\setminus}\serifGamma(\serifsigma)$ that coincides with the oriented angle between $\dot{\serifgamma}(\serifsigma)$ and $x-\serifgamma(\serifsigma)$, and vanishes on the points of the form $x=\serifgamma(\serifsigma)+\eps \dot{\serifgamma}(\serifsigma)$  for $\eps>0$ small enough. 
\end{proposition}
\begin{proof}
Since the connected components of $B{\setminus}\serifGamma (\serifsigma)$ have Lipschitz boundary, the space $L^{1,2}(B{\setminus}\serifGamma (\serifsigma))$ coincides with $H^1(B{\setminus}\serifGamma (\serifsigma))$. Then the proof of \eqref{definition-kappa} can be found in \cite[Theorem 4.4.3.7 and Section 5.2]{Gri85} and \cite[Appendix 1]{Mum-Sha}.
\end{proof}
\par
The  constant   $\kappa$ is proportional to the  stress intensity factor  considered in the engineering literature. It is related to the derivative of the energy with respect to the crack length, as we shall see in Proposition~\ref{prop:derivative-Stab} below.
\par
Given an open subset  $A\subset \Omega$ with Lipschitz boundary, a compact set $\sfK \subset \overline \Omega$, and a
function $\sfg\colon  \partial A \setminus \sfK\to \R $,  we define
  \begin{equation}
  \label{e-tilde}
  \begin{aligned}
\widetilde{\calE}(A;g,K): =   \min_{v\in \wAdm {g}{K}{A}}  \int_{A{\setminus}K} \tfrac12  |\nabla v|^2 \dd x\,,
\end{aligned}
 \end{equation}
where 
\begin{equation}
\label{def-admis-u-new}
 \wAdm {\sfg}{\sfK}{A}: = \{ v \in L^{1,2}(A{\setminus}\sfK)\, :  \ v = \sfg \  \ \text{on }  \partial A \setminus \sfK\}\,.
\end{equation}
\EEE

The following result  can be obtained by adapting the proof of \cite[Thm.\ 6.4.1]{Gri}.
\begin{proposition}
\label{prop:derivative-Stab}
Let  $B$ and $\gamma$ be as in Proposition \ref{definition-kappa} and let  $\sfg: \partial B {\setminus}\{\serifgamma(\serifsigma_0)\} \to \R$ be a function. For every $\serifsigma\in (\serifsigma_0,\serifsigma_1)$ suppose that
$\wAdm {\sfg}{\serifGamma(\serifsigma)}{B}\neq\emptyset$ and let 
$
\sfu(\serifsigma) \in \argmin_{ v\in  \wAdm {\sfg}{\serifGamma(\serifsigma)}{B}} \int_{B{\setminus}\serifGamma(\serifsigma)} \tfrac12  |\nabla v |^2 \dd x \,.
$
Then,
\begin{equation}
\label{dsigma-Y}
\frac{\dd}{\dd\serifsigma}  \widetilde{\calE}(B;g,\Gamma(\sigma)) = -\kappa(\sfu(\serifsigma) , \serifsigma)^2 \qquad \text{for every } \serifsigma \in (\serifsigma_0,\serifsigma_1)\,,
\end{equation}
with $\kappa$ defined by \eqref{definition-kappa}.
\end{proposition}
\paragraph{\bf Localization of the stability condition}
We  now prove that  the $\VE$-stability inequality  can be localized, in the spirit of 
 \cite[Lemma 8.5]{DMToa02}. 
 \begin{lemma}
\label{lemma:stab-local}
Assume that $(t,K)\in [0,T]\times \Xamb$ is $\cmdn$-stable and let
$
u \in  \argmin_{v\in \Adm{g(t)}{K}} \int_{\Omega{\setminus}K} \tfrac12 |\nabla  v|^2 \dd x \,.
$
Then, for every open subset $A\subset \Omega$ with Lipschitz boundary 
 we have 
\begin{equation}
\label{localized-stability}
\widetilde{\calE}(A;u,K) \leq  \widetilde{\calE}(A;u,K')  + \huno(K'{\setminus}K) + \int_{K'{\setminus}K} \!\!\!\!\! \!\!\!\!\! \dist{x}{K{\cap}A} \dd x 
 +(\lambda+\mu) \icost{K{\cap}\overline{A}}{(K'{\cup}K){\cap}\overline{A}}
\end{equation}
for all $K' \in \mathcal{K}(\overline{A})$ with $K' \supset K \cap \overline{A}$.
\end{lemma}
\begin{proof}
 Let $K' \in \mathcal{K}(\overline{A}) $. 
It follows from \eqref{Qstable}, with $K'$ replaced by $K'{\cup}K$, that
\begin{equation}
\label{stability-enlarged}
\ene t{ K} \leq  \ene t{ K'{\cup}K}   
+ \huno(K'{\setminus}K)  + \int_{ K'{\setminus}K} \dist{x}{ K } \dd x 
+(\lambda+\mu)\icost{ K}{K'{\cup} K} \,. 
\end{equation}
We repeat the very same calculations as in the proof of  \cite[Lemma 8.5]{DMToa02}, obtaining that 
\begin{equation}
\label{arte12}
\ene t{K'{\cup} K} -\ene t{ K}  \leq   \widetilde{\calE}(A; u,K')  -\widetilde{\calE}(A; u,K) \,. 
\end{equation}
As for the third term on the right-hand side of \eqref{stability-enlarged}, we observe that
\begin{equation}
\label{arte13}
\int_{K'{\setminus}K} \dist{x}{ K} \dd x   \leq \int_{K'{\setminus}K} \dist{x}{ K{\cap}A} \dd x .
\end{equation}
Finally, let us prove that  
\begin{equation}
\label{inequality alpha} 
\icost{ K}{K'{\cup} K}  \leq \icost{ K{\cap}\overline{A}}{(K'{\cup} K){\cap}\overline{A}} \,.
\end{equation}
It is enough to show that every connected 
component of $K'{\cup} K$ disjoint from  $ K$ is a connected 
component of $(K'{\cup} K){\cap}\overline{A}$. 
If $C$ is a connected component of $K'{\cup}K$ and does not intersect $K$,
then $C\subset K'\subset \overline{A}$, hence $C\subset (K'{\cup} K){\cap}\overline{A}$.
If $C'$ is a connected set such that $C\subset C'\subset (K'{\cup} K){\cap}\overline{A}$, 
then $C'\subset K'{\cup} K$ and hence $C'=C$. This shows that 
$C$ is a connected component of $(K'{\cup} K){\cap}\overline{A}$  and concludes the proof of
\eqref{inequality alpha}, which, together with \eqref{stability-enlarged}--\eqref{arte12}, yields \eqref{localized-stability}.
\end{proof}
\paragraph{\bf  Griffith's condition at the crack tips}
Our result for 
a $\VE$ solution $K\colon [0,T]\to \Xamb$ satisfying, in addition, the structural condition  stated in  Hypothesis \ref{h:griffith}, 
involves the constants the $\kappa_i= \kappa_i(u(t),\sigma_i(t))$  satisfying \eqref{definition-kappa} at the tips $\gamma_i(\sigma_i(t))$ of the branches of the  crack, where
$u(t)$ is  the corresponding minimal displacement (cf.\ \eqref{2.4bis} below).
\begin{theorem}
\label{th:griffith}
Let $K\colon [0,T]\to \Xamb$ be  a $\VE$ solution of the system for brittle fracture $\CRIS$, with time-dependent boundary datum $g\in \rmC^1([0,T];H^1(\Omega))$, and for every $t\in [0,T]$ let 
\begin{equation}
\label{2.4bis}
u(t) \in \argmin_{v \in \Adm{g(t)}{K(t)}} \int_{\Omega{\setminus}K(t)} \tfrac12 |\nabla  v|^2 \dd x \,.
\end{equation}
Assume that $K$
satisfies Hypothesis \ref{h:griffith} on some $(\tau_0,\tau_1)\subset [0,T]$, with arcs $\Gamma_i$ and
functions $\sigma_i$, $i=1,\ldots, p$.
Then,
\begin{subequations}
\label{Griff-crit}
\begin{align}
\label{Griff-crit--a}
&
\dot{\sigma}_i(t) \geq  0 && \foraa\, t \in (\tau_0,\tau_1),
\\
\label{Griff-crit--b}
& 1-\kappa_i (u(t),\sigma_i(t))^2 \geq 0 && \text{for all } 
 t \in (\tau_0,\tau_1),
\\
\label{Griff-crit--c}
& \left(1-\kappa_i (u(t),\sigma_i(t))^2\right)  \dot{\sigma}_i(t) =   0 && \foraa\, t \in (\tau_0,\tau_1)
\end{align}
\end{subequations}
for every $i=1,\ldots, p$.
\end{theorem}
\begin{remark}
\upshape
\label{rmk:Griffith}
Following \cite{DMToa02}, we observe that \eqref{Griff-crit--a} states that the length of every branch of the crack is non-decreasing, in accordance with  the irreversibility of the crack growth process;  \eqref{Griff-crit--b} imposes that the absolute value of the stress intensity factor, at each tip, be less or equal than $1$; by 
\eqref{Griff-crit--c}, the stress intensity factor reaches the threshold  values $\pm1$ as soon as the tip moves with positive velocity. 
In fact,  conditions \eqref{Griff-crit} rephrase  Griffith's criterion in our  context.
\par
 Therefore, Theorem \ref{th:griffith} ensures that  a  $\VE$ solution complying with Hyp.\ \ref{h:griffith}   satisfies
Griffith's criterion \EEE in the interval $(\tau_0,\tau_1)$ during which it evolves \emph{continuously} as a function of time,  like it  happens  for  the  quasistatic evolutions considered in 
 \cite[Thm.\ 8.4]{DMToa02}.  This is consistent with the fact that  the most relevant difference between $\VE$ solutions and quasistatic evolutions resides in the  jump behavior, as highlighted by Proposition~\ref{prop:charact-VE}. 
\end{remark}
\begin{proof}[Proof of Theorem \ref{th:griffith}]
As in the proof of  \cite[Thm.\ 8.4]{DMToa02}, we fix an arbitrary $t\in (\tau_0,\tau_1)$ and consider a family of open balls $B_1,\ldots, B_p$ centered at the points $\gamma_i(\sigma_i(t))$. Up to choosing their radii sufficiently small, we have that $\overline{B}_i \subset \Omega$ and $\overline{B}_i \cap K(\tau_0) = \overline{B}_i \cap \overline{B}_j = \overline{B}_i \cap \Gamma_j= \emptyset$ for $j\neq i$. Furthermore, we may assume that, for every $i=1,\ldots, p$,  
\[
B_i \cap \Gamma_i = \{ \gamma_i(\sigma)\, : \rho_i^0 <\sigma<\rho_i^1\} 
\]
for suitable $\rho_i^0$ and $\rho_i^1$ such that $\sigma_i^0<\rho_i^0 <\sigma_i(t)<\rho_i^1 <\sigma_i^1$, and that the arcs $\Gamma_i$ intersect $\partial B_i$ only at the points $\gamma_i(\rho_i^0)$ and $\gamma_i(\rho_i^1)$ with a transversal intersection.  Then, taking into account Hypothesis \ref{h:griffith},  we conclude that 
\begin{equation}
\label{8.4-1}
\overline{B}_i \cap K(s) = \overline{B}_i \cap \Gamma_i(\sigma_i(s)) = \{ \gamma_i(\sigma)\, : \ \rho_i^0 \leq \sigma \leq \sigma_i(s)\}
\end{equation}
whenever $\sigma_i(s) \in (\rho_i^0,\rho_i^1)$. In particular, \eqref{8.4-1} holds at $s=t$ and for $s$ sufficiently close to $t$,  since  $\sigma_i$ is continuous at $t$. 
\par
It follows from Lemma \ref{lemma:stab-local} that, for every $i=1,\ldots, p$, 
\[
\widetilde{\calE}(B_i;u(t),K(t)) \leq   \widetilde{\calE}(B_i;u(t),K')  + \huno(K'{\setminus}K(t)) + \int_{ K'{\setminus}K(t)} \!\!\!\!\! \!\!\!\!\!\dist{x}{K(t){\cap}B_i} \dd x 
 + (\lambda+\mu)\icost{K(t){\cap}\overline{B}_i}{(K'{\cup}K(t)){\cap}\overline{B}_i}  
\]
\ for all $K' \in \mathcal{K}(\overline{B}_i)$ with $K' \supset K(t){\cap}\overline{B}_i$, where  $\widetilde{\calE}$ is the localized energy functional  defined in  \eqref{e-tilde}.
Choosing $K'= \Gamma_i(\sigma) \cap \overline{B}_i =  \{ \gamma_i(\rho)\, : \ \rho_i^0 \leq \rho \leq \sigma\}$ \ with 
$\sigma \in [\sigma_i(t), \rho_i^1]$,
and recalling that $\huno(\Gamma_i(\sigma){\setminus}\Gamma_i(\sigma_i(t)))=\sigma-\sigma_i(t)$  
and $\icost{\Gamma_i(\sigma_i(t)){\cap}\overline{B}_i}{ \Gamma_i(\sigma){\cap}\overline{B}_i}=0$,
we deduce that 
\begin{equation}
\label{towards-1-der}
\begin{aligned}
\widetilde{\calE}(B_i;u(t),\Gamma_i(\sigma_i(t)) &\leq   \widetilde{\calE}(B_i;u(t),\Gamma_i(\sigma))  + 
\sigma-\sigma_i(t)+ 
\int_{\Gamma_i(\sigma){\setminus}\Gamma_i(\sigma_i(t))} \!\!\!\!\! \!\!\!\!\! \!\!\!\!\! \!\!\!\!\! \!\!\!\!\! 
\dist{x}{\Gamma_i(\sigma_i(t)){\cap}B_i} \dd x. 
 \end{aligned}
\end{equation}
 for all $\sigma \in [\sigma_i(t), \rho_i^1]$. 
Taking into account that 
\[ 
\lim_{\sigma \to \sigma_i(t)} 
\frac1{\sigma{-}\sigma_i(t)}  \int_{\Gamma_i(\sigma){\setminus}\Gamma_i(\sigma_i(t))} \!\!\!\!\! \!\!\!\!\! \!\!\!\!\! \!\!\!\!\! \!\!\!\!\! 
\dist{x}{\Gamma_i(\sigma_i(t)){\cap}B_i} \dd x  = 0,
\]
 from \eqref{towards-1-der}  we  obtain  that 
\[
\frac{\dd}{\dd\sigma}  \widetilde{\calE}(B_i; u(t),\Gamma_i(\sigma)) \big |_{\sigma=\sigma_i(t)} +1 \geq 0 \qquad \text{for all } i =1, \ldots, p.
\]
Then, \eqref{Griff-crit--b} follows from Proposition \ref{prop:derivative-Stab} applied with $\sfg=u(t)$.
\par
For every $[s,t] \subset (\tau_0, \tau_1)$ we have that 
$\icost{K(s)}{K(t)} = 0$ by Hypothesis \ref{h:griffith}. Hence $ \Vari{\icostname} Kst =0$, so that $ \Vari{\mdn} Kst =\huno(K(t){\setminus}K(s))$. 
Since $K$ evolves continuously in time on the interval $(\tau_0,\tau_1)$, the energy-dissipation balance \eqref{enbal-VE}
 reduces to 
\begin{equation}
\label{reduced-enbal}
\ene t{K(t)} +\huno(K(t){\setminus}K(s)) = \ene s{K(s)}+\int_s^t \partial_t \ene r{K(r)} \dd r \quad \text{for all } [s,t]\subset (\tau_0,\tau_1)\,.
\end{equation}
From \eqref{reduced-enbal}, with the very same arguments as in the proof of \cite[Thm.\ 8.4]{DMToa02} we deduce  \eqref{Griff-crit--c}. 
\end{proof}

\section{Extension to $2$D linearized elasticity}
\label{s:7}
In \cite{Chambolle03} the existence of  \emph{quasistatic evolutions} for fracture,  proved in the scalar setting in \cite{DMToa02}, was extended to the vectorial, still two-dimensional, setting of linearized elasticity. The argument relied on a density result of $H^1(A; \R^2)$-fields in the space of fields whose symmetrized gradient is in $L^2(A;\R_{\mathrm{sym}}^{2\times2})$, proved by the author in the case $A \subset \R^2$ is a bounded open set whose complement has a finite number of connected components.
\par
We will now briefly explain how the arguments in  \cite{Chambolle03} also allow us to prove the existence of visco-energetic solutions for the  \emph{vectorial} ($2$D, linearized elasticity) version of the system for brittle fracture, that we  address 
in a domain $\Omega \subset \R^2$ still complying with the conditions expounded at the beginning of Section \ref{s:2}. 
 The viscously corrected system for brittle fracture is now given by the quadruple $\CRISEL$ in which
\begin{itemize}
\item[-] 
the dissipation quasi-distance $\mdn$, and the viscous correction $\delta$ are still given by 
\eqref{dissipation-quasidistance}, and \eqref{delta}, respectively;
\item[-] the driving energy functional $\calE_{\mathrm{LE}} [0,T] \times \Xamb \to [0,+\infty)$  is defined by 
\begin{equation}
\label{energy-el}
\calE_{\mathrm{LE}}(t,K) :=  
\min_{u\in \Admle {g(t)}K}  \int_{\Omega{\setminus}K} \tfrac12  \mathbb{C} e(u) : e(u) \dd x,
\end{equation}
where
$\mathbb{C}$ is   the elasticity tensor and $e(u)$ denotes the symmetric part of $\nabla u$, 
\[
g\in \mathrm{C}^1([0,T];H^1(\Omega;\R^2)),
\] and 
the space for admissible displacements   is   now given by 
\[
 \Admle {g(t)}K: = \{ v \in \mathrm{LD}(\Omega{\setminus}K)\, :  \ v = g(t) \  \text{ on } \partial_D \Omega \setminus K\}\,.
\]
Here, following \cite{Chambolle03}, 
for a given $A \subset \R^2$ we denote by $\mathrm{LD}(A)$ the space
\[
\mathrm{LD}(A)  := \{ v \in L^2_{\mathrm{loc}}(A;\R^2)\, : \ e(v) \in L^2(A;\R_{\mathrm{sym}}^{2\times2})\}\,.
\]
\end{itemize}
We will denote by $\calF_{\mathrm{LE}}$ the functional associated with $\calE_{\mathrm{LE}}$ and a reference \REVIS crack set  $K_o$ (which may be again chosen as the empty set),   \EEE
 as in  \eqref{perturbed-energy}. 
\par
As we have seen in Section \ref{s:proof}, in order to prove the existence of $\VE$ solutions it is sufficient to show that  the system for brittle fracture $\CRISEL$ complies with conditions 
$\mathbf{<A>}$,  $\mathbf{<B>}$, and $\mathbf{<C>}$ listed  at the beginning of Section \ref{s:5}. Now, the viscous correction $\delta$ obviously still enjoys property $\mathbf{<B>}$.
As for $\mathbf{<A>}$, it follows from the following analogue of Proposition \ref{l:ad-A}. 

\begin{proposition}
\label{l:ad-A-le} 
The functional $\calE_{\mathrm{LE}}\colon [0,T]\times \Xamb \to [0,+\infty)$   defined in \eqref{energy-el} 
is  continuous  w.r.t.\ the $\htop_\R$-topology  on sublevels of the functional $\calF_{\mathrm{LE}}$ 
Moreover, $\partial_t \calE_{\mathrm{LE}}\colon  [0,T] \times \Xamb \to \R$  is given by 
\[
\partial_t \calE_{\mathrm{LE}}(t,K)=  \int_{\Omega{\setminus}K} \mathbb{C}e(u): e( \dot{g}(t)) \dd x
\] 
\REVIS with $u \in \Admle {g(t)}K$ a solution of the minimum problem in \eqref{energy-el};
$\partial_t \calE_{\mathrm{LE}}$ as well 
 is \EEE  continuous w.r.t.\ the $\htop_\R$-topology   on sublevels of $\calF_{\mathrm{LE}}$, \EEE and   fulfills estimate \eqref{A}.
\end{proposition}
The \emph{proof} of Proposition \ref{l:ad-A-le} follows from the arguments in \cite[Thm.\ 3]{Chambolle03}. 
Finally, Proposition \ref{l:ad-C}, guaranteeing the validity of  property $\mathbf{<C>}$, carries over to the present setting: in particular, 
the construction of the \REVIS mutual recovery sequence \EEE  $(K_n)_n$ fulfilling \eqref{rec-seq} developed throughout Section \ref{ss:5.1}
is still appropriate for this vectorial setting
 thanks to the aforementioned continuity properties of $\calE_{\mathrm{LE}}$.
\par
That is why the analogue of our existence Theorem \ref{thm:existVEcrack} holds for the system $\CRISEL$.
\section*{Acknowledgments} This paper is based on work supported by the National Research Projects (PRIN 2017) ``Variational Methods for Stationary and Evolution Problems with Singularities and Interfaces” (G.D.M. and R.T.) and  ``Gradient flows, Optimal Transport and Metric Measure Structures” (R.R. and G.S.).

G.S.~acknowledges the support of the Institute of Advanced Study of the Technical University of Munich.
G.S.~and R.R.~acknowledge also the support of the IMATI-CNR, Pavia. 

The authors are members of the Gruppo Nazionale  per l'Analisi Matematica, la Probabilit\`a e loro Applicazioni
(GNAMPA) of the Istituto Nazionale di Alta Matematica (INdAM).

\REVISDOUBT Finally, the authors are grateful to an anonymous referee for reading the paper very thoroughly and for several insightful suggestions. \EEE


\end{document}